\documentclass[10pt,reqno]{amsart}
\usepackage{eucal,times,fullpage,amsmath,amsthm,amssymb,mathrsfs,stmaryrd,color}
\usepackage[all]{xy}
\usepackage{url}

\newcommand{\calHom}{\mathcal{H}\mathrm{om}}

\newcommand{\G}{\mathbf{G}}
\newcommand{\Z}{\mathbf{Z}}

\newcommand{\Q}{\mathbf{Q}}
\renewcommand{\P}{\mathbf{P}}

\newcommand{\Spec}{{\operatorname{Spec}}}
\newcommand{\Supp}{{\operatorname{Supp}}}
\newcommand{\uSpec}{{\underline{\Spec}}}

\newcommand{\Fib}{\mathrm{Fib}}

\newcommand{\Hom}{\operatorname{Hom}}

\newcommand{\PGL}{\operatorname{PGL}}
\newcommand{\Ext}{\operatorname{Ext}}

\newcommand{\ob}{\mathrm{ob}}

\newcommand{\ch}{\mathrm{char}}
\newcommand{\D}{\mathrm{D}}
\newcommand{\R}{\mathrm{R}}

\newcommand{\id}{\operatorname{id}}

\newcommand{\red}{\mathrm{red}}

\newcommand{\fm}{\mathfrak{m}}
\newcommand{\fn}{\mathfrak{n}}
\newcommand{\fp}{\mathfrak{p}}

\newcommand{\Def}{{\operatorname{Def}}}
\newcommand{\Art}{{\operatorname{Art}}}
\newcommand{\SArt}{{\operatorname{SArt}}}
\newcommand{\Mod}{{\operatorname{Mod}}}
\newcommand{\Der}{{\operatorname{Der}}}
\newcommand{\QCoh}{{\operatorname{QCoh}}}
\newcommand{\Alg}{ {\operatorname{Alg}}}
\newcommand{\SAlg}{ {\operatorname{SAlg}}}
\newcommand{\can}{ {\mathrm{can}}}
\newcommand{\depth}{\operatorname{depth}}

\newcommand{\factor}[2]{\left. \raise 2pt\hbox{\ensuremath{#1}} \right/
        \hskip -2pt\raise -2pt\hbox{\ensuremath{#2}}}
\DeclareMathOperator{\Ric}{Ric}

\newcommand{\shE}{\mathcal{E}}
\newcommand{\shF}{\mathcal{F}}
\newcommand{\shG}{\mathcal{G}}
\newcommand{\shT}{\mathcal{T}}
\newcommand{\Xt}{\tilde{X}}
\DeclareMathOperator{\Prod}{Prod}
\newcommand{\Aut}{\operatorname{Aut}}
\newcommand{\rk}{\operatorname{rk}}

\newcommand{\comment}[1]{}

\begin{document}

\bibliographystyle{alpha}

\newtheorem{theorem}{Theorem}[section]
\newtheorem*{theorem*}{Theorem}
\newtheorem*{condition*}{Condition}
\newtheorem*{definition*}{Definition}
\newtheorem{proposition}[theorem]{Proposition}
\newtheorem{lemma}[theorem]{Lemma}
\newtheorem{corollary}[theorem]{Corollary}
\newtheorem{claim}[theorem]{Claim}
\newtheorem{claimex}{Claim}[theorem]

\theoremstyle{definition}
\newtheorem{definition}[theorem]{Definition}
\newtheorem{question}[theorem]{Question}
\newtheorem{remark}[theorem]{Remark}
\newtheorem{construction}[theorem]{Construction}
\newtheorem{example}[theorem]{Example}
\newtheorem{condition}[theorem]{Condition}
\newtheorem{warning}[theorem]{Warning}
\newtheorem{notation}[theorem]{Notation}

\title{Moduli of products of stable varieties}
\author{Bhargav Bhatt, Wei Ho, Zsolt Patakfalvi, Christian Schnell}
\begin{abstract}
We study the moduli space of a product of stable varieties over the field of complex numbers, as defined via the minimal model program. Our main results are: (a) taking products gives a well-defined morphism from the product of moduli spaces of stable varieties to the moduli space of a product of stable varieties, (b) this map is always finite \'etale, and (c) this map very often is an isomorphism. Our results generalize and complete the work of Van Opstall in dimension $1$. The local results rely on a study of the cotangent complex using some derived algebro-geometric methods, while the global ones use some differential-geometric input.
\end{abstract}
\maketitle


\section{Introduction}
\label{sec:intro}

The moduli space $\mathcal{M}_g$ of curves and its Deligne-Mumford compactification $\overline{\mathcal{M}_g}$ are two fundamental objects of modern mathematics with wide-ranging applications. A key to their utility is the modularity of the compactification $\overline{\mathcal{M}_g}$: the compactification itself parametrizes curves, possibly with mild singularities.  Recent advances in the minimal model program \cite{BC_CP_HCD_MJ_EOM} have provided us with a good higher dimensional analogue of this phenomenon: after fixing the necessary numerical invariants, one now has access to a compact moduli space $\overline{\mathcal{M}_h}$ that contains the space $\mathcal{M}_h$ of smooth objects as an open subspace, with the space $\overline{\mathcal{M}_h}$ itself parametrizing mildly singular varieties called {\em stable} varieties \cite{VE_QPM, KJ_SO, KollarModuliSurvey}. Although $\overline{\mathcal{M}_h}$ shares many nice properties of $\overline{\mathcal{M}_g}$, e.g., it is a DM-stack of finite type over the base field, it may possibly have many connected components that behave very differently \cite{CF_CCO, VR_MLI}. Hence, almost all available results on the global geometry of $\overline{\mathcal{M}_h}$ pertain to  specific components of the moduli of surfaces (e.g., \cite{VanOpstallModProd, VOMA_SDIP, LW_SD, RS_CM, AV_PR_ECO, LY_ACO}) or special components of the moduli of log-stable varieties (e.g., \cite{HP_KS_TJ_COT, HP_KS_TJ_SPT, AV_CMI, HB_SLS, HP_CM}).

Our goal in this paper is to produce results applicable to every component of $\overline{\mathcal{M}_h}$ for any $h$ --- and in particular, to any dimension --- by generalizing the work of Van Opstall \cite{VanOpstallModProd}. 
Specifically, we explore the behavior of these moduli spaces under the operation of taking products. To explain our results, let us fix some notation first (precise definitions will be given later). Let $k$ be a field of characteristic $0$.  Given a stable variety $Z$ over $k$, let $\mathcal{M}(Z)$ denote the connected component of the appropriate moduli space $\overline{\mathcal{M}_h}$ spanned by $Z$; this space is a Deligne-Mumford stack.  Given stable varieties $X$ and $Y$, we show that taking products defines a morphism 
\[ \Prod_{X,Y}:\mathcal{M}(X) \times \mathcal{M}(Y) \to \mathcal{M}(X \times Y).\] 
Our main local result is

\begin{theorem}
\label{mainthm:local}
The map $\Prod_{X,Y}$ is a finite \'etale cover of Deligne-Mumford stacks for any stable varieties $X$ and $Y$.
\end{theorem}
Going one step further, one may ask when the map $\Prod_{X,Y}$ is an isomorphism. By Theorem \ref{mainthm:local}, it suffices to find a {\em single point} of $\mathcal{M}(X \times Y)$ where $\Prod_{X,Y}$ has degree $1$. If $X$ and $Y$ are isomorphic or even simply deformation equivalent, then $\Prod_{X,Y}$ cannot be an isomorphism due to the symmetry of the source. Our main global result is that this is essentially the only obstruction, provided we work with {\em smooth} varieties.

\begin{theorem}
\label{mainthm:global}
Let $X$ and $Y$ be two stable varieties such that $X \times Y$ is smooth and neither
$X$ nor $Y$ can be written nontrivially as a product of two stable varieties.  If $X$ and $Y$ are not deformation equivalent, then $\Prod_{X,Y}$ is an isomorphism.  Otherwise, the map $\Prod_{X,Y}$ is an $S_2$-torsor.
\end{theorem}

The (slightly technical) notion of deformation equivalence above will be discussed more carefully in \S \ref{sec:definitions}.  A generalization of Theorem \ref{mainthm:global} applies to smooth stable varieties admitting a product decomposition, as explained in Theorem \ref{theorem:main}.  We expect but do not know if these results are true without the smoothness assumption.

Theorem \ref{mainthm:global} is a consequence of the following more general result
about canonically polarized manifolds (i.e., compact complex manifolds with ample
canonical bundle), whose proof occupies \S\ref{sec:globaltheory} below.

\begin{theorem}
\label{mainthm:productirreducible}
Every canonically polarized manifold decomposes uniquely into
a product of irreducible factors.
\end{theorem}

\subsection*{Comments on proofs}
Granting the existence of a proper moduli stack of stable varieties, Theorem \ref{mainthm:local} immediately reduces to a statement about the deformation theory of stable varieties. We approach this statement via the Abramovich-Hassett theory of canonical covering stacks which relates the {\em admissible} deformation theory of a stable variety $X$ with the usual deformation theory of an associated stack $X^\can$. The key point then (following obvious notation) is to show that $\Def_{X^\can} \times \Def_{Y^\can}$ is equivalent to $\Def_{ (X \times Y)^\can }$; we show this by equating both sides with $\Def_{X^\can \times Y^\can}$ via a detailed study of the relevant cotangent complexes. 

Theorem \ref{mainthm:global} and Theorem \ref{mainthm:productirreducible} are proven
using differential-geometric methods. The main input is the polystability of the
tangent bundle on a canonically polarized manifold (ensured by two theorems of Yau
and Uhlenbeck-Yau), and the fact that a direct sum decomposition of the tangent
bundle induces a product decomposition of the universal cover (from a theorem of Beauville).

\subsection*{Organization of the paper}

We set up the problem at hand, in \S \ref{sec:stablevarieties}, by describing the appropriate moduli functors for families of stable varieties and constructing the product map.  Theorem \ref{mainthm:local} is proven in \S \ref{sec:localtheory} by first considering a general theorem about deformations of products in \S \ref{sec:defthyabstractprod} and then specializing to our moduli spaces in \S \ref{sec:qgordefthy}.  These proofs use the language of derived algebraic geometry, which is reviewed in Appendix \ref{sec:dag}. Finally, \S \ref{sec:globaltheory} explains how many ways stable varieties can decompose as products, under assumptions about smoothness; in the many cases for which stable varieties decompose uniquely as products, the associated product map is injective.

\subsection*{Notation}

Throughout this paper, we use $k$ to denote a field of characteristic $0$ with two exceptions: in \S \ref{sec:localtheory}, we allow $k$ to have positive characteristic unless otherwise indicated,  and in Appendix \ref{sec:dag}, we allow  $k$ to be an arbitrary ring.

\subsection*{Acknowledgements} We thank Dan Abramovich for suggesting the directions pursued here, and the AMS and NSF for providing wonderful working conditions at the MRC program at Snowbird in July 2010, where this project was initiated. 
In addition, we would like to thank Stefan Kebekus and Chenyang Xu for help with proving Lemma \ref{stablevarinfaut}, and Zhiyu Tian for working with us when the project started.


\section{Stable varieties and construction of the product map} \label{sec:stablevarieties}

In this section, we define the moduli functor parametrizing stable varieties, and show that it is representable by a proper Deligne-Mumford stack; the properness uses recent results in the minimal model program due to Hacon-McKernan-Xu (unpublished). We then show there is a well-defined product map, which we will investigate in the sequel.

\subsection{Definitions of stable varieties and moduli functors} \label{sec:definitions}

As stated before, the moduli space of smooth projective curves of genus at least $2$ may be compactified by adding points representing some mildly singular curves obtained from smooth curves by a limiting procedure; the resulting curves are called {\em stable curves}. To compactify the space of birational equivalence classes of  varieties of general type in higher dimensions, one is then confronted with the problem of determining the singular varieties that should be allowed at the boundary. Mori theory solves this problem by providing a viable candidate definition for higher dimensional {\em stable varieties} and {\em stable families}; the robustness of the solution ensures that the moduli functor thus defined is automatically separated (by an old result of \cite{MatsusakaMumford}) and also proper, granting standard conjectures in higher dimensional geometry that are now theorems.

Our goal in this section is to review the definitions of stable varieties and stable families,  and also to say a few words about the resulting moduli space; more information can be found in the survey articles \cite{KovacsYPG,KollarModuliSurvey}. First, we recall some basic definitions. A variety $X$ is said to have {\em log canonical singularities} if $X$ is normal, $\Q$-Gorenstein, and satisfies the following: for a log resolution of singularities $g: \widetilde{X} \to X$ with exceptional divisor $E = \cup_i E_i$, if we write $K_{\widetilde{X}} = g^* K_X + \sum_i a_i E_i$, then we have $a_i \geq -1$ for all $i$.  The notion of {\em semi-log canonical singularities} is a non-normal generalization of log canonical singularities. Its definition is almost verbatim the same as of log canonical, but the log resolution is replaced by a good semi-resolution. We refer the reader to \cite[\S 6.5]{KovacsYPG} and \cite[Definition-Lemma 5.1]{KJ_SO} for more, and simply remark here that such singularities are automatically reduced, satisfy Serre's $S_2$ condition, are $\Q$-Gorenstein, and are Gorenstein in codimension $1$. For a coherent sheaf $\mathcal{F}$ on a noetherian scheme $X$ such that $\Supp(\mathcal{F}) = X$,  the {\em reflexive hull} $\mathcal{F}^{\ast\ast}$ is defined to be the double dual of $\mathcal{F}$. If $X$ is $S_2$ and $G_1$ (Gorenstein in codimension one) and $\mathcal{F}$ is a line bundle in codimension one, say over $U \subseteq X$, then $\mathcal{F}^{**} \cong j_* (\mathcal{F}|_U)$ \cite[Theorem 1.12]{HR_GD}, where $j : U \to X$ is the natural embedding. The {\em reflexive powers} $\mathcal{F}^{[i]}$ are then defined to be $\big(\mathcal{F}^{\otimes i}\big)^{\ast\ast}$ for any integer $i$ with the convention that $\mathcal{F}^{\otimes i} := \calHom(\mathcal{F},\mathcal{O}_X)^{\otimes -i}$ for $i < 0$; these definitions will typically be applied when $\mathcal{F}$ has generic rank $1$.

The main object of study in this paper is contained in the following definition:

\begin{definition}
A proper geometrically connected $k$-variety $X$ is called {\em stable} if $X$ has semi-log canonical singularities and $K_X$ is a $\Q$-Cartier and ample divisor.
\end{definition}

Next, we define families. The naive definition of a family of stable varieties (namely, a flat family with stable fibers) leads to pathologies as there are ``too many'' such families; see \cite[\S 7]{KovacsYPG} for examples. The correct definition, given below, imposes certain global constraints on the family.  In the sequel, we sometimes refer to such families as {\em admissible} families. The condition appearing below is known as {\em Koll\'ar's condition}.

\begin{definition} \label{def:stablefamily}
Given a $k$-scheme $S$, a {\em stable family} $f: X \to S$ is a proper flat morphism whose fibers are stable varieties and such that $\omega_{X/S}^{[m]}$ is flat over $S$ and commutes with base change, for every $m \in \Z$.
\end{definition}

Finally, we are ready to define the moduli functor of stable varieties.  The functor $SlcMod_h$ in \cite{KollarModuliSurvey} is similar but we choose to keep track of automorphisms.

\begin{definition}
Let $h(m)$ be an integer-valued function.  The moduli functor $\overline{\mathcal{M}_h}$ of stable varieties with Hilbert function $h$ is defined by setting $\overline{\mathcal{M}_h}(S)$ to be the groupoid of stable families $f: X \to S$ whose fibers have Hilbert function $h$ with respect to $\omega_{X/S}$. Given a stable variety $X$ over $k$, we let $\mathcal{M}(X)$ denote the connected component of $\overline{\mathcal{M}_{h(X)}}$ that contains $[X]$, where $h(X)$ is the Hilbert function of $X$.  Then two varieties $X$ and $Y$ are {\em deformation equivalent} if $\mathcal{M}(X)$ and $\mathcal{M}(Y)$ coincide.
\end{definition}

\subsection{Automorphisms of stable varieties}

In this section, we show that $\mathcal{M}(X)$ is a Deligne-Mumford stack for stable varieties $X$, although this fact must surely be known by the experts.  We start with a lemma that bounds how negative the canonical line bundle on a resolution of singularities of a stable variety can be.

\begin{lemma}
\label{lem:lc-bigness}
Let $X$ be a stable variety over $k$, and let $\pi:Y \to X$ be a semi-resolution with (reduced) exceptional divisor $E$. Then $K_Y + E$ is big.
\end{lemma}
\begin{proof}
Let $E = \sum_i E_i$ be the reduced union of the $\pi$-exceptional divisors. As $X$ has semi-log canonical singularities, we can write
\[ K_Y = \pi^* K_X + \sum_i a_i E_i \]
with $a_i \geq -1$, or equivalently, we can write
\[ K_Y + E = \pi^* K_X + \sum_i b_i E_i\]
with $b_i \geq 0$. The stability of $X$ implies that $K_X$ is ample. The preceding formula then expresses $K_Y + E$ as  the sum of a big divisor and an effective one, proving bigness.
\end{proof}

We now show that stable varieties do not admit infinitesimal automorphisms; this fact was stated in \cite{VanOpstallModProd}, but the proof was incomplete.

\begin{lemma} \label{stablevarinfaut}
Let $X$ be a stable variety over a field $k$ of characteristic $0$. Then $X$ has no infinitesimal automorphisms.
\end{lemma}

We give two proofs of this result: the first is cohomological and relies on recent work \cite{GD_KS_KSJ_PT_DF}.

\begin{proof}[Proof 1]
We wish to show that
$\Hom_X( L_X, \mathcal{O}_X)=0$. Consider the usual exact triangle
\begin{equation*}
\xymatrix{
\tau_{< 0} L_X \ar[r] & L_X \ar[r] & \Omega_X \ar[r]^-{+1} &
}
\end{equation*}
As $\Ext^i_X(\tau_{< 0} L_X, \mathcal{O}_X) = 0$ for $i = 0, -1$, one has $\Hom_X( L_X, \mathcal{O}_X) \simeq \Hom_X( \Omega_X, \mathcal{O}_X)$, so it is enough to show that the latter is zero.  We will show the vanishing of this group when $X$ is normal; the general case is similar but requires an analysis of how $\Omega_{\overline{X}}(\log D)$ relates to $\Omega_X$, where $\mathrm{n} :\overline{X} \to X$ is the normalization of $X$ and $D$ is the divisor of the double locus of $\mathrm{n}$.  Because restriction to the smooth locus is fully faithful on the category of reflexive sheaves on a normal scheme, we have
\begin{equation*}
 \Hom_X(\Omega_X, \mathcal{O}_X) \simeq
\Hom_{X_\mathrm{sm}}(\Omega_X, \mathcal{O}_X) \simeq
\Hom_{X_\mathrm{sm}} (\omega_X, \Omega_X^{n-1})\simeq
\Hom_X(\omega_X, \Omega_X^{[n-1]}),
\end{equation*}
which vanishes by \cite[Theorem 7.2]{GD_KS_KSJ_PT_DF}.
\end{proof}

The second proof of Lemma \ref{stablevarinfaut} is more direct and geometric.

\begin{proof}[Proof 2]
We give a proof in the case that $X$ is normal and of index $1$, leaving the rest for the reader. Since $X$ is assumed to have an ample canonical bundle, the group sheaf $T \mapsto \Aut(X_T)$ is represented by a closed subgroup scheme $\underline{\Aut}(X) \subset \PGL_n$ for suitable $n$, which allows us to talk about its identity component $\underline{\Aut}^0(X)$. Now assume towards contradiction that $X$ has nontrivial infinitesimal automorphisms, i.e., that $\underline{\Aut}^0(X)$ has a nonzero tangent space at the identity. By Chevalley's theorem, $\underline{\Aut}^0(X)$ either contains a linear algebraic subgroup, or is itself an abelian variety. We treat these cases separately; the idea in either case is to show that the presence of a positive dimensional group action forces $X$ to be fibered over a lower dimensional base with fibers of Kodaira dimension $\leq 0$ (up to an alteration), which is then shown to contradict stability.

Assume first that $\underline{\Aut}^0(X)$ has a nonzero linear algebraic subgroup. Since $\ch(k) = 0$, we can pick a one-dimensional connected smooth group scheme $G \subset \underline{\Aut}^0(X)$, necessarily either $\G_m$ or $\G_a$. Let $Z \subset X$ denote the singular locus, and choose a $G$-equivariant resolution of singularities $f:Y \to X$ with exceptional locus $E = \pi^{-1}(Z)_\red$. Now consider the diagram
\[ \xymatrix{ G \times Y \ar[r]^a \ar[d]^\pi & Y \\ Y } \]
where $a$ is the map defining the group action, while $\pi$ is a projection map. Since the representation $G \to \underline{\Aut}(Y)$ is faithful with $\dim(G) > 0$ and $G$ is smooth, we can choose a smooth divisor $H \subset Y$ such that the restriction of $a$ to $G \times H$ is dominant and generically \'etale. By compactifying $\pi|_H$ and resolving singularities, we obtain a diagram
\[ \xymatrix{ G \times H  \ar@{^{(}->}[r]^j \ar[d]^{\pi|_H} & \mathcal{C}  \ar[d]^{\overline{\pi}} \ar[r]^{q} & Y \\
				H \ar@{=}[r] & H & } \]
where $\mathcal{C}$ is smooth, $\overline{\pi}$ is a proper surjective morphism of relative dimension $1$, $j$ is a dense open immersion, and $q$ is a proper, surjective, generically \'etale map extending $a$. In particular, the map $\overline{\pi}$ restricts to the trivial $\overline{G}$-bundle over some dense open subset of $H$, where $\overline{G} \simeq \P^1$ is the natural projective compactification of $G$. We then have the following possibilities for $G$ and the corresponding intersection numbers of $\omega_{\mathcal{C}}(q^{-1}E)$ with a general fiber of $\overline{\pi}$.
\begin{itemize}
\item $G = \G_m$: The general fiber of $\overline{\pi}$ is a $\P^1$ that passes through a general point of $Y$ and meets $E$ in at most two points: its image in $X$ contains the $\G_m$-orbit through a smooth point, and hence meets $\mathrm{Sing}(X) = f(E)$ in at most $2$ points. Since $\omega_{\mathcal{C}}$ restricts to $\mathcal{O}_{\P^1}(-2)$ on the general fiber of $\overline{\pi}$, we find that $\omega_{\mathcal{C}}(q^{-1} E)$ has degree $\leq 0$ on a general fiber of $\overline{\pi}$.
\item $G = \G_a$: Exactly as above, we find that $\omega_{\mathcal{C}}(q^{-1}E)$ has degree $\leq -1$ on a general fiber of $\overline{\pi}$.
\end{itemize}
Hence, the bundle $\omega_{\mathcal{C}}(q^{-1}E)$ always has degree $\leq 0$ on a general fiber of $\overline{\pi}$. On the other hand, Lemma \ref{lem:lc-bigness} shows that $\omega_Y(E)$ is big, and thus so is  $\omega_{\mathcal{C}}(q^{-1} E) = q^* \omega_Y(E) \otimes \omega_{\mathcal{C}/Y}$ since $\omega_{\mathcal{C}/Y}$ is effective. Hence, the degree of $\omega_{\mathcal{C}}(q^{-1}E)$ on a general fiber of $\overline{\pi}$ should be positive, which leads to a contradiction proving the claim in this case.

If $\underline{\Aut}^0(X)$ does not contain a linear algebraic subgroup, then $\underline{\Aut}^0(X)$ is an abelian variety $A$, say of dimension $g$. Assume first that $g \geq \dim(X)$. Since $A$ acts faithfully on $X$, there is an open subset $U \subset X$ which consists of points with no $A$-stabilizers. Translating a closed point in $U$ using $A$ then shows that $g = \dim(X)$, and that $X = A$. In particular, $\omega_X$ is trivial, contradicting  the ampleness of $\omega_X$. If $g < \dim(X)$, then we argue as in the case of $\G_m$ above, subject to the following changes: use a codimension $g$ subvariety $H \subset Y$ instead of a divisor; use that restriction of $\omega_{\mathcal{C}}$ to the appropriate general fiber is trivial, as abelian varieties have trivial canonical bundle; and observe that the general fibers of $\overline{\pi}$ map to $A$-orbits of smooth points in $X$, and so miss $E$ entirely when mapped to $Y$.
\end{proof}

\begin{remark}
Lemma \ref{stablevarinfaut} admits a simple cohomological proof when $X$ is itself smooth: it suffices to check that $H^0(X,T_X) = 0$, which follows by Serre duality and Kodaira vanishing (using the ampleness of $\omega_X$). An advantage of the cohomological approach is that it also works in characteristic $p$ as long as $X$ lifts to $W_2$ and $\dim(X) < p$. We do not know what happens if either of these assumptions is dropped. The geometric argument in the second proof of Lemma \ref{stablevarinfaut} runs into problems immediately as infinitesimal group actions cannot usually be integrated to positive dimensional group actions in positive characteristic.
\end{remark}

Next, we prove a separation result for the moduli functor:

\begin{lemma}
\label{lemma_separable}
Let $X \to S$ and $Y \to S$ be two families of stable schemes over a curve $S$ with normal generic fiber. Let $0 \in S$ be a point, such that for $U:=S \setminus \{0\}$, $X_U \cong Y_U$ as schemes over $U$. Then, $X \cong Y$ as schemes over $S$.
\end{lemma}

\begin{proof}
First, we may assume that $S$ is affine, by throwing out a point if necessary. Choose a common resolution $Z$ of $X$ and $Y$. Since $X_U \cong Y_U$, $Z$ can be chosen so that it is an isomorphism over $U$. Let $f : Z \to X$ and $g : Z \to Y$ the birational morphisms obtained this way. Since $X$ and $Y$ are families of stable schemes over a smooth curve, they are $S_2$ by \cite[Proposition 6.3.1]{GA_EGA_IV_II}. Since both have normal generic fiber, both $X$ and $Y$ are $R_1$. Hence, they are both normal. Also, since both $\omega_{X/U}$ and $\omega_{Y/U}$ are $\Q$-line bundles, so are $\omega_X$ and $\omega_Y$. Therefore, by \cite[Theorem]{KM_IOA}, $(X,X_s)$ and $(Y,Y_s)$ are log canonical for every $s \in S$. In particular, so are $X$ and $Y$, and furthermore, every divisor with negative discrepancy dominates $S$. That is, the canonical divisors of $X$, $Y$ and $Z$ are related by the equations
\begin{equation}
\label{eq:common_resolution_canonical_divisors}
 K_Z + M  = f^* K_X + F \qquad K_Z + N = g^* K_Y + G,
\end{equation}
where $F$, $G$, $M$ and $N$  effective, exceptional (with respect to the adequate morphisms) $\Q$-divisors, such that every prime divisor in $ M $ and $ N $ has coefficient at most $1$ and dominates $S$. Furthermore, since $M$ and $N$ are determined on $X_U$ and $Y_U$, in fact $M=N$. We use $M$ to denote both divisors.

Let $r$ be the lowest common multiple of the indices of $K_X$ and $K_Y$. Let $R(X, D):= \bigoplus_j H^0(X, \mathcal{O}_X(jD))$ denote the Cox ring of the divisor $D$ on $X$, for any divisor $D$ on a scheme $X$.  Then by \eqref{eq:common_resolution_canonical_divisors}, we have
\begin{equation*}
R(X,r K_X) \simeq R(Z, r(g^* K_Y + G)) \simeq R(Z,r(K_Z + M)) \simeq R(Z,r(f^*K_X + F)) \simeq R(Y, r K_Y)
\end{equation*}
where the first isomorphism follows from $F$ being effective and $f$-exceptional (similarly for the last isomorphism, using $G$).
Since both $rK_X$ and $rK_Y$ are ample line bundles (as they are relatively ample over an affine base), we obtain $$X \simeq \mathrm{Proj} \ R(X, rK_X) \simeq \mathrm{Proj} \  R(Y, rK_Y) \simeq Y.$$ Furthermore, since $S$ was affine, this isomorphism respects $S$.
\end{proof}

The main existence result concerning the moduli functor is

\begin{theorem}
\label{thm:proper_DM} 
For a fixed Hilbert function $h$, the moduli functor $\overline{\mathcal{M}_h}$ is a proper Deligne-Mumford stack.
\end{theorem}
\begin{proof}
That $\overline{\mathcal{M}_h}$ is a locally algebraic Artin stack follows from \cite{AbramovichHassett} and Artin's method. Lemma \ref{stablevarinfaut} then shows that $\overline{\mathcal{M}_h}$ is actually a Deligne-Mumford stack. For generically normal families, the separatedness of $\overline{\mathcal{M}_h}$ follows from Lemma \ref{lemma_separable}, and the general case can be proven by similar techniques. Finally, properness follows from recently announced results of Hacon-McKernan-Xu (unpublished).
\end{proof}

\subsection{The stability of products}

This section is devoted to constructing the map $\Prod_{X,Y}$ alluded to in  \S \ref{sec:intro}. First, we check that a product of stable varieties is stable.

\begin{lemma} 
\label{lem:slc}
The product of varieties with only semi-log canonical singularities has semi-log
canonical singularities.
\end{lemma}

\begin{proof}
This is proved in \cite[Theorem~3.2]{VanOpstallModProd}, but we give another argument here. We
use the criterion that $X$ has semi-log canonical singularities if and only if the pair
$(X', D)$ is log canonical, where $X' \to X$ is the normalization
and $D$ the double point divisor.

Let $X_1$ and $X_2$ be two varieties with only semi-log canonical singularities, and set
$X = X_1 \times X_2$. Then we have $X' = X_1' \times X_2'$, and $D = (D_1 \times X_2)
\cup (X_1 \times D_2)$, and therefore $(X', D) = (X_1', D_1) \times (X_2', D_2)$. By
assumption, $X_1$ and $X_2$ are reduced and $\Q$-Gorenstein, and so the same is
clearly true for $X'$. Now let $f_i \colon Y_i \to X_i'$ be log resolutions for the
two pairs; by assumption on the singularities,
\[
	K_{Y_i} \equiv f_i^{\ast} \bigl( K_{X_i'} + D_i \bigr) + \sum_j a_{i,j} E_{i,j}
\]
with $a_{i,j} \geq -1$. Setting $Y = Y_1 \times Y_2$ and $f = f_1 \times f_2$, 
the morphism $f \colon Y \to X'$ is a log resolution for the pair
$(X', D)$. We compute that
\[
	K_Y \equiv p_1^{\ast} K_{Y_1} + p_2^{\ast} K_{Y_2} 
		\equiv f^{\ast} \bigl( K_{X'} + D \bigr) 
			+ \sum_j \bigl( a_{1,j} E_{1,j} \times Y_2 + a_{2,j} Y_1 \times E_{2,j} \bigr),
\]
which shows that $(X', D)$ is indeed log canonical.
\end{proof}

If $\mathcal{F}$ is a sheaf on $X$, we let $\mathcal{F}^* = \calHom_{\mathcal{O}_X}(\mathcal{F}, \mathcal{O}_X)$ denote
the $\mathcal{O}_X$-linear dual. We record an elementary algebraic fact next that will be used repeatedly in the sequel.

\begin{lemma}
\label{lem:reflex-pullback}
Let $f:X \to S$ be a flat morphism of noetherian schemes, and let $\mathcal{F}$ be a coherent sheaf on $S$. If $\mathcal{F}$ is reflexive, so is  $f^* \mathcal{F}$. If $f$ is surjective, the converse is also true.
\end{lemma}
\begin{proof}
The formation of $\calHom_S(\mathcal{E},\mathcal{G})$ commutes with flat base change on $S$ for any pair of coherent sheaves $\mathcal{E}$ and $\mathcal{G}$. In particular, the formation of $\mathcal{F}^*$ commutes with flat base change. Now consider the biduality map $\mathcal{F} \to (\mathcal{F}^*)^*$. Since the reflexivity of $\mathcal{F}$ is precisely the condition that this map is an isomorphism, all claims follow from basic properties of flatness.
\end{proof}

Next, we show that exterior products of reflexive sheaves remain reflexive.

\begin{lemma} \label{lem:reflexive}
Let $f \colon X \to B$ and $g \colon Y \to B$ be two flat morphisms, and let $Z = X
\times_B Y$ be their fiber product. Let $\mathcal{F}$ and $\mathcal{G}$ be reflexive sheaves on $X$
and $Y$, respectively. If $\mathcal{F}$ and $\mathcal{G}$ are $B$-flat, then $\mathcal{F} \boxtimes \mathcal{G}$ is a reflexive sheaf on $Z$.
\end{lemma}
\begin{proof}
Let $p_X:Z \to X$ and $p_Y:Z \to Y$ be the two projection maps.  By Lemma \ref{lem:reflex-pullback}, the sheaves $p_X^* \mathcal{F}$ and $p_Y^* \mathcal{G}$ are reflexive. Then we have
\begin{align*}
(p_X^* \mathcal{F}^* \otimes p_Y^* \mathcal{G}^*)^*
&= \calHom(p_X^* \mathcal{F}^* \otimes p_Y^* \mathcal{G}^*,\mathcal{O}_Z) \\
&\simeq \calHom( p_X^* \mathcal{F}^*, \calHom(p_Y^* \mathcal{G}^*, \mathcal{O}_Z) )  \qquad (\textrm{by adjunction})\\
&\simeq \calHom( p_X^* \mathcal{F}^*, p_Y^* \mathcal{G}) \qquad (\textrm{by reflexivity of } \mathcal{G})  \\
&\simeq \calHom( p_X^* \mathcal{F}^*, \mathcal{O}_Z) \otimes p_Y^* \mathcal{G} \qquad (\textrm{by flatness of } p_X \textrm{ and } p_Y )\\
&\simeq p_X^* (\mathcal{F}^*)^* \otimes p_Y^* \mathcal{G}  \\
&\simeq p_X^* \mathcal{F} \otimes p_Y^* \mathcal{G} \qquad (\textrm{by reflexivity of } \mathcal{F}).
\end{align*}
Thus, $p_X^* \mathcal{F} \otimes p_Y^* \mathcal{G}$ is the dual of a coherent sheaf on $Z$ and, therefore, reflexive.
\end{proof}

We now show that the product of stable families is stable.

\begin{proposition} \label{prop:product}
The fiber product of two stable families is again a stable family.
\end{proposition}

\begin{proof}
Let $f \colon X \to B$ and $g \colon Y \to B$ be two stable families, and set $Z = X
\times_B Y$ and $h \colon Z \to B$. Since $f$ and $g$ are flat, projective, and have
connected fibers, the same is true for $h$. Lemma~\ref{lem:slc} shows that each fiber
$Z_b = X_b \times Y_b$ has semi-log canonical singularities. Next, we verify
Koll\'ar's condition. By assumption, the formation of $\omega_{X/B}^{[k]}$ commutes
with arbitrary base change, and so by \cite[Theorem~5.1.4]{AbramovichHassett}, we may conclude that $\omega_{X/B}^{[k]}$ is flat over $B$; we also reproduce the essential part of this argument below as Lemma \ref{modreflexflat} and Corollary \ref{reflexbcflat}  for the convenience of the reader. Since $f$ and
$g$ are flat morphisms, Lemma~\ref{lem:reflexive} shows that
\[
	p_X^{\ast} \omega_{X/B}^{[k]} \otimes p_Y^{\ast} \omega_{Y/B}^{[k]}
\]
is again a reflexive sheaf on $Z$. Arguing as in \cite[Lemma~7.3]{KovacsYPG}, we see that it agrees with the
reflexive sheaf $\omega_{Z/B}^{[k]}$ on an open set whose complement has relative codimension
at least two in $Z$. We must therefore have
\[
	\omega_{Z/B}^{[k]} 
		\simeq p_X^{\ast} \omega_{X/B}^{[k]} \otimes p_Y^{\ast} \omega_{Y/B}^{[k]}.
\]
This formula implies that the formation of $\omega_{Z/B}^{[k]}$ commutes with
arbitrary base change, and so Koll\'ar's condition holds for the family $h \colon Z
\to B$. Also, when $k$ is the least common multiple of the index of $X$ and the index
of $Y$, the formula shows that $\omega_{Z/B}^{[k]}$ is a relatively
ample line bundle, proving that $\omega_{Z/B}$ is an ample $\Q$-line bundle. This
concludes the proof that $h \colon Z \to B$ is a stable family.
\end{proof}

The next lemma and following corollary are here for the reader's convenience, as they are used in the proof of Proposition \ref{prop:product}; see \cite{KollarFlatness} for more results like these.

\begin{lemma} 
\label{modreflexflat}
Let $f:(R,\fm) \to (S,\fn)$ be an essentially finitely presented flat local map of noetherian local rings. Let $M$ be a finitely presented $S$-module. Assume the following:
\begin{enumerate}
\item The locus of points on $\Spec(S)$ where $M$ is flat over $\Spec(R)$ is dense in the fiber $\Spec(S/\fm)$.
\item The support of any nonzero $m \in M/\fm M$ contains a generic point of $\Spec(S/\fm)$.
\end{enumerate}
Then $M$ is $R$-flat.
\end{lemma}
\begin{proof}
By the local flatness criterion, it suffices to check that the natural surjective maps
\[ a_n: \fm^n / \fm^{n+1} \otimes_{R/\fm} M / \fm M \twoheadrightarrow \fm^{n} M / \fm^{n+1} M \]
are isomorphisms for all $n$. Let $K_n = \ker(a_n)$. The assumption that the flat locus is dense in the fibers tells us that $K_n$ is not supported at any of the generic points of $\Spec(S/\fm)$. Since the source of $a_n$ can be identified with a direct sum of copies of $M/\fm M$, it follows that if $K_n \neq 0$, then $M / \fm M$ admits sections not supported at the generic points of $\Spec(S/\fm)$. However, this contradicts the second assumption, so $K_n = 0$, proving flatness.
\end{proof}

\begin{corollary}
\label{reflexbcflat}
Let $f:X \to S$ be a locally finitely presented flat map of noetherian schemes with fibers that are $S_2$ and of pure dimension $d$. Let $U \subset X$ be an open subset dense in all the fibers.  Let $\mathcal{F}$ be a coherent sheaf on $X$ such that $\mathcal{F}|_U$ is $S$-flat. Assume that $\mathcal{F}|_{X_s}$ is reflexive for any $s \in S$. Then $\mathcal{F}$ is $S$-flat.
\end{corollary}
\begin{proof}
There is nothing to show when $d = 0$ as $U = X$ in that case by density, so we may assume $d > 0$. To show the $S$-flatness of $\mathcal{F}$, we will check that the conditions of Lemma \ref{modreflexflat} hold locally on $X$. The first condition is satisfied by assumption on $U$. For the second one, given a point $s \in S$, the reflexivity of $\mathcal{F}|_{X_s}$ tells us that, locally on $X_s$, we may realize $\mathcal{F}|_{X_s}$ as a subsheaf of a direct sum of copies of $\mathcal{O}_{X_s}$. Since $X_s$ is a pure and positive dimensional  $S_2$ scheme,  all nonzero local sections of $\mathcal{O}_{X_s}$ are supported at some generic point of $X_s$, and so the same is true for $\mathcal{F}|_{X_s}$, showing the second condition is satisfied. By Lemma \ref{modreflexflat}, we conclude that $\mathcal{F}$ is $S$-flat, as desired.
\end{proof}

By Proposition \ref{prop:product}, the fiber product of two stable families is also a stable family. Hence, we define the desired product map as follows:

\begin{definition}
For any two stable varieties $X$ and $Y$, let $\mathrm{Prod_{X,Y}}$ be the morphism
\begin{equation*}
 \Prod_{X,Y} : \mathcal{M}(X) \times \mathcal{M}(Y) \to \mathcal{M}(X \times Y).
\end{equation*}
defined by taking fiber products of stable families.
\end{definition}


\section{The local theory} \label{sec:localtheory}

Our goal in this section is to explain why taking products of stable varieties defines a finite \'etale morphism on moduli spaces. The two main steps of the proof are: (a) showing that the deformation theory of products behaves in the expected way for a fairly large class of algebro-geometric objects, and (b) dealing with the slightly subtle issues related to the deformation theory of stable varieties, stemming ultimately from Koll\'ar's condition in Definition \ref{def:stablefamily} of admissible stable families.  We first study (a) in \S \ref{sec:defthyabstractprod}.  Then \S \ref{sec:defmorphisms} contains some general results on deformations of morphism, which form the key technical ingredients of the proofs in \S \ref{sec:qgordefthy}, where we carry out step (b).  

The two main tools used in our proofs are the Abramovich-Hassett description of the admissible deformation theory of stable varieties in terms of the (usual) deformation theory of certain associated stacks (see \cite{AbramovichHassett}), and derived algebraic geometry. The former reduces the admissible deformation theory of stable varieties to the usual deformation theory of certain associated stacks, permitting us to use the cotangent complex. The main advantage of the derived perspective is an {\em explicit} construction of deformations and obstructions which makes calculations feasible, especially in the singular case (see the proof of Proposition \ref{deformds}). The relevant background is summarized in Appendix \ref{sec:dag}; we note here that all derived rings that occur in the discussion below are especially mild: they are simplicial $k$-algebras with finite dimensional homology.

\subsection{The deformation theory of products}
\label{sec:defthyabstractprod}

Fix a field $k$.   The main result of this section, Theorem \ref{mainthm:defthy}, is a general theorem about the deformation theory of products of two Deligne-Mumford stacks.  Under some mild hypotheses on the two stacks, the main one being lack of infinitesimal automorphisms, we show that the deformations of the product are given uniquely by products of deformations of the factors. The meat of the proof is a rather thankless task: we check that obstructions behave predictably under taking products.  We will use this result in \S \ref{sec:qgordefthy} to understand the infinitesimal behavior of our global product map $\Prod_{X,Y}$. We remark that the aforementioned stack-theoretic description of the admissible deformation theory of a stable variety necessitates formulating and proving results in the present section for stacks rather than varieties.

We introduce two pieces of notation first.

\begin{notation}
Let $\SArt_k$ denote the $\infty$-category of derived local artinian $k$-algebras, i.e., those $A \in \SAlg_k$ with $\pi_0(A)$ local with residue field $k$, and $\oplus_i \pi_i(A)$ finite dimensional as a $k$-vector space. The category $\SArt_k$ provides test objects for deformation-theoretic questions in derived algebraic geometry, and we call its objects small derived algebras. Any map $A \to B$ in $\SArt_k$ that is surjective on $\pi_0$ can be factored as $A = A_0 \to A_1 \to \dots \to A_n = B$ with $A_i \to A_{i+1}$ a square-zero extension of $A_{i+1}$ by $k[j]$ for some $j$ (see \cite[Lemma 6.2.6]{LurieDAG}). We let $\Art_k$ denote the full subcategory of $\SArt_k$ spanned by discrete small derived algebras. Note that $\Art_k$ is simply the ordinary category of artinian local $k$-algebras with residue field $k$; we refer to its objects as small algebras.
\end{notation}

\begin{notation}
For a Deligne-Mumford $k$-stack $X$, let $\Def_X$ be the $\infty$-groupoid-valued functor which associates to $A \in \SArt_k$ the $\infty$-groupoid of all pairs $(f:\mathcal{X} \to \Spec(A),i:X \to \mathcal{X})$ where $f$ is a flat morphism, and $i$ identifies $X$ with the special fiber of $f$. We will refer to such pairs $(f,i)$ as flat deformations of $X$. When restricted to $\Art_k \subset \SArt_k$, this definition recovers the ordinary groupoid-valued functor of flat deformations of $X$. For a morphism $\pi:Y \to X$, let $\Def_\pi(A)$ be the $\infty$-groupoid of quadruples $(f:\mathcal{X} \to \Spec(A),g:\mathcal{Y} \to \Spec(A),\pi_A:\mathcal{Y} \to \mathcal{X},\phi)$ where $f$ and $g$ are flat deformations of $X$ and $Y$ respectively to $A$,  $\pi_A$ is an  $A$-map deforming $\pi$, and $\phi$ is an identification of $\pi_A \otimes_A k$ with $\pi$.
\end{notation}

Given two Deligne-Mumford $k$-stacks $X$ and $Y$, there is a natural morphism 
\[ \mathfrak{p}\mathrm{rod}_{X,Y}:\Def_X \times \Def_Y \to \Def_{X \times Y} \]
given by taking fiber products. Our basic theorem concerns the behavior of the map $\mathfrak{p}\mathrm{rod}_{X,Y}$:

\begin{theorem}
\label{mainthm:defthy}
Fix a field $k$. Let $X$ and $Y$ be proper geometrically connected and geometrically reduced Deligne-Mumford $k$-stacks with no infinitesimal automorphisms. Then the map $\mathfrak{p}\mathrm{rod}_{X,Y}$ considered above is an isomorphism of functors on $\Art_k$.
\end{theorem}

Let us record certain vanishings that are available to us; these results enable fluid passage between the product and its factors.

\begin{lemma}
\label{cotpull}
Fix a field $k$. Let $X$ and $Y$ be proper Deligne-Mumford $k$-stacks. Assume that $X$ admits no infinitesimal automorphisms, and that $H^0(Y,\mathcal{O}_Y) = k$. Then the natural map
\[ \Ext_X^i(L_X,\mathcal{O}_X) \to \Ext^i_{X \times Y}({p_1}^*L_X,\mathcal{O}_{X \times Y}) \]
induced by pulling back along the projection $p_1:X \times Y \to X$ is bijective for $i = 0,1$, and injective for $i = 2$.
\end{lemma}
\begin{proof}
The projection formula and adjointness give natural identifications
\begin{eqnarray*} 
\Ext^i_{X \times Y}({p_1}^*L_X,\mathcal{O}_{X \times Y})  &=&  \Ext^i_X(L_X,\R {p_1}_* \mathcal{O}_{X \times Y}) \\
													&=& \Ext^i_X(L_X,\R \Gamma(Y,\mathcal{O}_Y) \otimes \mathcal{O}_X).
\end{eqnarray*}
Now consider the exact triangle
\[ \mathcal{O}_X \stackrel{u}{\to} \R\Gamma(Y,\mathcal{O}_Y) \otimes_k \mathcal{O}_X \to \mathcal{Q} \]
where $\mathcal{Q}$ is defined to be the homotopy cokernel of $u$. Applying $\Ext^i(L_X,-)$ gives 
\[ \Ext^i(L_X,\mathcal{Q}[-1]) \to \Ext^i(L_X,\mathcal{O}_X) \to  \Ext^i_{X \times Y}({p_1}^*L_X,\mathcal{O}_{X \times Y}) \to \Ext^i_X(L_X,\mathcal{Q}). \]
Thus, it suffices to check that $\Ext^i(L_X,\mathcal{Q}) = 0$ for $i \leq 1$. Since $L_X$ is connective and $\mathcal{Q} \in D^{\geq 1}(X)$, we immediately see that $\Ext^0(L_X,\mathcal{Q}) = 0$. To check that $\Ext^1(L_X,\mathcal{Q})$ vanishes as well, note that the exact triangle 
\[ \tau_{\geq 2} \mathcal{Q}[-1] \to \mathcal{H}^1(\mathcal{Q})[-1] \to \mathcal{Q} \]
shows that $\Ext^1(L_X,\mathcal{Q}) \simeq \Ext^0(L_X,\mathcal{H}^1(\mathcal{Q}))$. By construction, $\mathcal{H}^1(\mathcal{Q}) \simeq H^1(Y,\mathcal{O}_Y) \otimes \mathcal{O}_X$ is a free $\mathcal{O}_X$-module. The desired claim now follows from the stability assumption that $\Ext^0(L_X,\mathcal{O}_X) = 0$.
\end{proof}

We can now prove the desired result.

\begin{proof}[Proof of Theorem \ref{mainthm:defthy}]
We will show that 
\[ \mathfrak{p}\mathrm{rod}_{X,Y}(A):\Def_X(A) \times \Def_Y(A) \to \Def_{X \times Y}(A) \]
is an equivalence of groupoids for $A \in \Art_k$ by working inductively on $\dim_k(A)$. As $X$ and $Y$ lack infinitesimal automorphisms, the groupoids in question are discrete, and will be viewed as sets. When $\dim_k(A) = 1$, we have $A = k$ and there is nothing to show as both sides are reduced to points. By induction, we may assume that the desired claim is known for all $A \in \Art_k$ with $\dim_k(A) \leq n$ for some fixed integer $n$. Given an $\tilde{A} \in \Art_k$ with $\dim_k(\tilde{A}) = n + 1$, we can find a surjection $\tilde{A} \to A$ with kernel $k$ as an $A$-module. This gives a diagram 
\[ \xymatrix{ \Def_X(\tilde{A}) \times \Def_Y(\tilde{A}) \ar[rr]^-{\mathfrak{p}\mathrm{rod}_{X,Y}(\tilde{A})} \ar[d] && \Def_{X \times Y}(\tilde{A})\ar[d] \\
				\Def_X(A) \times \Def_Y(A) \ar[rr]^-{\mathfrak{p}\mathrm{rod}_{X,Y}(A)}& & \Def_{X \times Y}(A) } \]
with $\mathfrak{p}\mathrm{rod}_{X,Y}(A)$ bijective by induction. We will show that $\mathfrak{p}\mathrm{rod}_{X,Y}(\tilde{A})$ is bijective. As there is nothing to show if the bottom row is empty, we may fix a base point of the bottom row, i.e., we fix flat deformations $f:\mathcal{X} \to \Spec(A)$ and $g:\mathcal{Y} \to \Spec(A)$ of $X$ and $Y$ to $\Spec(A)$. Let $\pi_{f,g}:\mathcal{X} \times_{\Spec(A)} \mathcal{Y} \to \Spec(A)$ denote their fiber product, and let $p:\mathcal{X} \times_{\Spec(A)} \mathcal{Y} \to \mathcal{X}$ and $q:\mathcal{X} \times_{\Spec(A)} \mathcal{Y} \to \mathcal{Y}$ be the two projection maps.

We first show that all fibers of $\mathfrak{p}\mathrm{rod}_{X,Y}(\tilde{A})$ is non-empty, i.e., if $\pi_{f,g}$  admits a deformation across $\Spec(A) \hookrightarrow \Spec(A')$, then the same is true for $f$ and $g$. Let $D_A:L_A \to k[1]$ be the derivation classifying the surjection $\tilde{A} \to A$ (see Theorem \ref{thm:dag-sq-zero}). Associated to this derivation, we have obstruction classes 
\[ \ob(f,f^* D_A):L_{\mathcal{X}/A}[-1] \to \mathcal{O}_X[1] \quad \textrm{and} \quad \ob(g,g^* D_A):L_{\mathcal{Y}/A}[-1] \to \mathcal{O}_{Y}[1]  \]
on $\mathcal{X}$ and $\mathcal{Y}$, and the obstruction class
\[  \ob(\pi_{f,g},\pi_{f,g}^* D_A): L_{\mathcal{X} \times_{\Spec(A)} \mathcal{Y}/\Spec(A)}[-1] \to \mathcal{O}_{\mathcal{X} \times_{\Spec(A)} \mathcal{Y}}[1] \]
on the product given by Theorem \ref{thm:dag-sqzero-obs}. By Theorem \ref{thm:dag-obs-der-compat}, these classes are compatible in the sense that the following diagram commutes
\[ \xymatrix{ p^* L_{\mathcal{X}/A}[-1] \ar[rrrr]^-{p^* \ob(f,f^* D_A)} \ar[d]  & & & & \mathcal{O}_{\mathcal{X} \times_{\Spec(A)} \mathcal{Y}}[1] \ar@{=}[d] \\
			  L_{\mathcal{X} \times_{\Spec(A)} \mathcal{Y}/\Spec(A)}[-1] \ar[rrrr]^-{\ob(\pi_{f,g},\pi_{f,g}^* D_A)} & & & & \mathcal{O}_{\mathcal{X} \times_{\Spec(A)} \mathcal{Y}}[1]  \\
			  q^* L_{\mathcal{Y}/A}[-1] \ar[rrrr]^-{q^* \ob(g,g^* D_A)} \ar[u]  & & & & \mathcal{O}_{\mathcal{X} \times_{\Spec(A)} \mathcal{Y}}[1] \ar@{=}[u]. } \]
The assumption that $\pi_{f,g}$ admits a deformation across $\Spec(A) \hookrightarrow \Spec(A')$ ensures that the middle horizontal arrow in the above diagram is $0$. It follows by the commutativity that the same is true for other horizontal arrows, i.e., that $p^* \ob(f,f^* D_A) = 0$, and similarly for $Y$. To show that $\ob(f,f^* D_A) = 0$, it now suffices to show that the pullback
\[ \pi_0(\Hom_{\mathcal{X}}(L_{\mathcal{X}/\Spec(A)}[-1],k \otimes_A \mathcal{O}_\mathcal{X}[1])) \to \pi_0(\Hom_{\mathcal{X} \times_{\Spec(A)} \mathcal{Y}}(p_1^*(L_{\mathcal{X}/\Spec(A)})[-1],p_1^*(k \otimes_A \mathcal{O}_{\mathcal{X}})[1])) \]
is injective, and similarly for $Y$. Simplifying, this amounts to showing that the pullback
\[ \Ext^2_\mathcal{X}(L_{\mathcal{X}/\Spec(A)},\mathcal{O}_X) \to \Ext^2_{\mathcal{X} \times_{\Spec(A)} \mathcal{Y}}(p_1^*L_{\mathcal{X}/\Spec(A)},\mathcal{O}_{X \times Y}) \]
is injective, and similarly for $Y$. By base change (see \S \ref{derbasechange}) and adjunction, it is enough to check that the pullback
\[ \Ext^2_X(L_X,\mathcal{O}_X) \to \Ext^2_{X \times Y}(p_1^*L_X,\mathcal{O}_{X \times Y})\]
is injective, which follows from Lemma \ref{cotpull}; similarly for $Y$.

Next, we show that all fibers of $\mathfrak{p}\mathrm{rod}_{X,Y}(\tilde{A})$ are reduced to point, i.e., we will check that all possible deformations of $\mathcal{X} \times_{\Spec(A)} \mathcal{Y} \to \Spec(A)$ across $\Spec(A) \hookrightarrow \Spec(A')$ are obtained {\em uniquely} by taking products of deformations of each factor. By the above, we may assume that both $\mathcal{X} \to \Spec(A)$ and $\mathcal{Y} \to \Spec(A)$ admit deformations across $\Spec(A) \hookrightarrow \Spec(A')$. Following the same method used above to linearize the problem, we immediately reduce to verifying that the natural map
\[ \Ext^1_X(L_X,\mathcal{O}_X) \times \Ext^1_Y(L_Y,\mathcal{O}_Y) \to \Ext^1(L_{X \times Y},\mathcal{O}_{X \times Y}) \]
is bijective. This, in turn, results from the base change formula (see \S \ref{derbasechange}) and Lemma \ref{cotpull}.
\end{proof}

\begin{warning}
The conclusion of Theorem \ref{mainthm:defthy} fails if we consider both sides as functors on the larger category $\SArt_k$ of all small derived algebras rather than simply the ordinary ones. Indeed, the data of the functor $\Def_X$ on $\SArt_k$ is equivalent to the data of the object $\R\Hom(L_X,\mathcal{O}_X)$ (with its extra structure coming from Lie theory; see \cite[Theorem 5.2]{LurieICM}) at least in characteristic $0$. The failure of the product map 
 \[ \R\Hom(L_X,\mathcal{O}_X) \times \R\Hom(L_Y,\mathcal{O}_Y) \to \R\Hom(L_{X \times Y},\mathcal{O}_{X \times Y}) \]
to be an isomorphism then explains the failure of $\mathfrak{p}\mathrm{rod}_{X,Y}$ to be an equivalence as functors on $\SArt_k$. For example, let $X$ and $Y$ be genus $g$ curves for $g > 0$. One then computes that $\Ext^2(L_{X \times Y},\mathcal{O}_{X \times Y}) \neq 0$, but $\Ext^2(L_X,\mathcal{O}_X) = \Ext^2(L_X,\mathcal{O}_Y) = 0$. What this means is that $X \times Y$ has a nontrivial deformation over the derived local artian $k$-algebra $k \oplus k[1]$, while $X$ and $Y$ do not.
\end{warning}

\begin{remark}
The main input from derived algebraic geometry in our proof of Theorem \ref{mainthm:defthy} is an explicit construction of the deformation and obstruction classes associated to a morphism $\pi:Y \to X$; having access to the construction renders the functoriality transparent. It is tempting to deduce this functoriality directly from Illusie's formula for the obstruction class in terms of the cup product of the Kodaira-Spencer class for $\pi$ and the $\Ext^1$ class describing the relevant deformation of $X$. One can implement this strategy with a good understanding of the functoriality of the $\Ext^1$ class describing the relevant deformation of $X$.
\end{remark}

\begin{remark}
The proof of Theorem \ref{mainthm:defthy} has two essential parts: showing that the map $\mathfrak{p}\mathrm{rod}_{X,Y}$ is injective, and showing that $\mathfrak{p}\mathrm{rod}_{X,Y}$ is surjective. The injectivity of $\mathfrak{p}\mathrm{rod}_{X,Y}$ is a standard verification with tangent spaces that holds under fairly general hypothesis. The surjectivity of $\mathfrak{p}\mathrm{rod}_{X,Y}$, on the other hand, crucially needs the stability assumption that $X$ and $Y$ have no infinitesimal automorphisms. For example, if $X$ and $Y$ are elliptic curves, then the product variety $X \times Y$ admits a $4$-dimensional space of first order deformations, while the first order deformations which are products span a $2$-dimensional subspace (and both sides are unobstructed).
\end{remark}

\subsection{Some general results on deformations of morphisms}
\label{sec:defmorphisms}

The general theme of the results discussed in this section is the deformation theory of morphisms. Our goal is to write down some natural conditions on a morphism $\pi:Y \to X$ which allow one to transfer deformation-theoretic information from $X$ to $Y$, and vice versa. These results constitute the heart of the proof of Proposition \ref{defthyindqcov} in \S \ref{sec:qgordefthy}, but may be read independently of the rest of the paper.

We first need the following algebraic lemma:

\begin{lemma}
\label{gorextvan}
Let $R$ be a noetherian ring, and let $M$ be a finitely generated $R$-module. If $M$ vanishes at all points of codimension $\leq N$ of $\Spec(R)$ and $R$ satisfies Serre's condition $S_{N+1}$ at all points of $\Supp(M)$,  then $\Ext^i_R(M,R) = 0$ for $0 \leq i \leq N$.
\end{lemma}
\begin{proof}
Let $d = \dim(R)$, let $X = \Spec(R)$, let $\mathcal{F}$ be the coherent $\mathcal{O}_X$-module defined by $M$, let $Z = \Supp(\mathcal{F})$, let $U = X  - Z$, and let $j:U \to \Spec(R)$ be the natural open immersion. Then we have the exact triangle
\[ \mathcal{O}_X \to \R j_* \mathcal{O}_U \to \mathcal{Q} \]
where $\mathcal{Q}$ is the homotopy cokernel, and is identified with the complex $\R\Gamma_Z(\mathcal{O}_X)[1]$. Applying $\R \Hom(\mathcal{F},-)$ and taking homology, we obtain a long exact sequence
\[ \dots \Ext_X^{i-1}(\mathcal{F},\mathcal{Q}) \to \Ext_X^i(\mathcal{F},\mathcal{O}_X)  \to \Ext_X^i(\mathcal{F},\R j_* \mathcal{O}_U) \dots . \]
The term on the right is $0$ by adjunction and the fact that $j^* \mathcal{F} = 0$. Hence, it suffices to show that $\Ext_X^{i-1}(\mathcal{F},\mathcal{Q}) = 0$ for $i \leq N$. By connectivity estimates, it suffices to check that  $\mathcal{Q}[N-1] \in \D^{[1,\infty]}(\mathcal{O}_X)$, i.e. it suffices to check that $\mathcal{H}^{i-1}(\mathcal{Q}) = 0$ for $i \leq N$. Translating to local cohomology, it suffices to check that $H^i_Z(\mathcal{O}_X) = 0$ for $i \leq N$. Since the codimension of any point occurring in $Z$ is at least $N+1$, the claim now follows from the assumption that $X$ satisfies Serre's condition $S_{N+1}$ at all points of $Z$ coupled with the fact that the $I$-depth of a module $P$ over a ring $R$ with ideal $I$ can be recovered as the infimum of the depths of the localizations of $P$ at all points of $\Spec(R/I)$; see, for example, \cite[\S 15, page 105]{MatCA}.
\end{proof}

The following proposition gives some conditions on a map $\pi:Y \to X$ which ensure that deformations of $X$ can be followed by deformations of $Y$ and that of $\pi$.

\begin{proposition}
\label{defcms}
Let $\pi:Y \to X$ be an essentially finitely presented morphism of noetherian Deligne-Mumford stacks. Assume that the following conditions hold:
\begin{enumerate}
\item The map  $\pi$ is \'etale on an open set $U \subset Y$ that contains all the codimension $\leq 2$ points of $Y$.
\item The stack $Y$ satisfies Serre's condition $S_3$ at points of $Z = Y - U$.
\end{enumerate}
Then $\Ext^i(L_\pi,\mathcal{O}_X) = 0$ for $i \leq 2$. If $X$ is essentially finitely presented over a field $k$, then the natural map $\Def_\pi \to \Def_X$ is an equivalence of functors on $\Art_k$.
\end{proposition}
\begin{proof}
We first show the $\Ext$ vanishing claim. By the local-to-global spectral sequence for $\Ext$, it suffices to show that $\mathcal{E}\mathrm{xt}^i(L_\pi,\mathcal{O}_Y) = 0$ for $i \leq 2$. Since the latter is a local statement, we may \'etale localize on $Y$ and reduce to the case that $Y$ is a noetherian local scheme. In this local setup, we will check that $\Ext^i(L_\pi,\mathcal{O}_Y) = 0$ for $i \leq 2$. We first filter $L_\pi$ using the filtration in the derived category arising from the standard $t$-structure. This filtration of $L_\pi$ has associated graded pieces of the form $\mathcal{H}^{-j}(L_\pi)[j]$. Hence, the groups $\Ext^i(L_\pi,\mathcal{O}_Y)$ are filtered with graded pieces contained in $\Ext^{i-j}(\mathcal{H}^{-j}(L_\pi),\mathcal{O}_Y)$ for $0 \leq j \leq i$. Thus, to show $\Ext^i(L_\pi,\mathcal{O}_Y) = 0$ for $i \leq 2$, it suffices to show that $\Ext^k(\mathcal{H}^{-j}(L_\pi),\mathcal{O}_Y) = 0$ for $0 \leq k \leq 2$, and any $j$. However, this follows from the $N = 2$ case of Lemma \ref{gorextvan} once we observe that the sheaves $\mathcal{H}^{-j}(L_\pi)$ vanish at all codimension $\leq 2$ points of $Y$ by the assumption that $\pi$ is \'etale at such points.

The claim about deformation functors is deduced in a standard manner from the relative $\Ext$ vanishing proven above. Fix an $A \in \Art_k$, and consider the induced map of groupoids $f:\Def_\pi(A) \to \Def_X(A)$. The vanishing of $\Ext^2(L_\pi,\mathcal{O}_Y)$ implies that $f$ is surjective on $\pi_0$, the vanishing of $\Ext^1(L_\pi,\mathcal{O}_Y)$ implies that $f$ is injective on $\pi_0$, and the vanishing of $\Ext^i(L_\pi,\mathcal{O}_Y)$ for $i \leq 1$ implies that $f$ is bijective on $\pi_1$. To make these assertions precise, one climbs up a tower of small extensions as in the proof of Theorem \ref{mainthm:defthy}; we leave the details to the reader.
\end{proof}

Next, we study the dual question of conditions on a map $\pi:Y \to X$ that ensure that deformations of $Y$ can be followed by deformations of $X$ and $\pi$. 

\begin{proposition}
\label{deformds}
Let $\pi:Y \to X$ be a morphism of essentially finitely presented Deligne-Mumford stacks over a field $k$ satisfying $\pi_* \mathcal{O}_Y \simeq \mathcal{O}_X$ and that $\Ext^0_X(\Omega_X, \R^1 \pi_* \mathcal{O}_Y) = 0$. Then the forgetful morphism $q:\Def_\pi \to \Def_Y$ of functors on $\Art_k$ is formally smooth with discrete fibers; it is an equivalence if $X$ has no infinitesimal automorphisms.
\end{proposition}

\begin{proof}

Let $f:X \to \Spec(k)$ and $g:Y \to \Spec(k)$ denote the structure maps. Fix an $A \in \Art_k$ and a flat deformation $\pi_A:\mathcal{Y} \to \mathcal{X}$ of $\pi$ to $\Spec(A)$. Given a surjection $A' \to A$ with kernel isomorphic to $k$, we obtain a diagram
\begin{equation}
\label{diag:defmor}
 \xymatrix{ \Def_\pi(A') \ar[r]^a \ar[d]^b & \Def_Y(A') \ar[d]^c \\
			  \Def_\pi(A) \ar[r]^d & \Def_Y(A).}
\end{equation}
By induction on $\dim_k(A)$, we may assume that $d$ is surjective on $\pi_0$ and has discrete fibers. Furthermore, if $X$ has no infinitesimal automorphisms we may also assume that $d$ is an equivalence. We will show the following: (a) $a$ is surjective on $\pi_0$ and has discrete fibers, and (b) $a$ is an equivalence if $X$ has no infinitesimal automorphisms.

Fix a flat deformation $\mathcal{Y}' \to \Spec(A')$ of $\mathcal{Y} \to \Spec(A)$ corresponding to a point $p_{Y'} \in \Def_Y(A')$.  Let  $\Fib(a,p_{Y'})$ denote the homotopy fiber of the map $a$ at the point $p_{Y'}$; this $\infty$-groupoid can be thought of as parametrizing triples $(\mathcal{X}' \to \Spec(A'), \pi_{A'}:\mathcal{Y}' \to \mathcal{X}',\phi)$ where $\mathcal{X}' \to \Spec(A')$ is a flat deformation of $X$ to $\Spec(A')$, $\pi_{A'}$ is a deformation of $\pi$ to $\Spec(A')$, and $\phi$ is an identification of the restriction $(\mathcal{X}',\pi_{A'})|_A$ with $(\mathcal{X},\pi_A)$. We will now check that $\Fib(a,p_{Y'})$ is discrete and non-empty, and furthermore it is contractible when $X$ has no infinitesimal automorphisms. First, we record a relation between maps on $X$ and $Y$:

\begin{claimex}
\label{claim:extpullback}
The natural map
\begin{equation*}
\Ext^i_\mathcal{X}(L_{\mathcal{X}/A},\mathcal{O}_X)) \to \Ext^i_\mathcal{Y}(\pi_A^* L_{\mathcal{X}/A},\mathcal{O}_Y)
\end{equation*}
is an isomorphisms for $i \leq 1$ and it is injective for $i=2$.
\end{claimex}

\begin{proof}[Proof of claim]The above natural map is obtained as the composition of the adjunction  $\Ext^i_\mathcal{Y}(\pi_A^* L_{\mathcal{X}/A},\mathcal{O}_Y) \cong \Ext^i_\mathcal{X}( L_{\mathcal{X}/A},\R \pi_{A,*}\mathcal{O}_Y)$, and the natural map $\Ext^i_\mathcal{X}( L_{\mathcal{X}/A},\mathcal{O}_X) \to \Ext^i_\mathcal{X}( L_{\mathcal{X}/A},\R \pi_{A,*}\mathcal{O}_Y)$. The former one is an isomorphism, hence we are supposed to prove the claimed properties only for the latter  maps. Consider the following exact triangle guaranteed by the condition $f_* \mathcal{O}_Y \cong \mathcal{O}_X$.
\begin{equation*}
\xymatrix{
 \mathcal{O}_X \ar[r] & \R \pi_{A,*} \mathcal{O}_Y \ar[r] & \tau_{\geq 1} \R \pi_{A,*} \mathcal{O}_Y 
}
\end{equation*}
Applying $\Ext^i_{\mathcal{X}}(L_{\mathcal{X}/A}, \_  )$ implies that it is enough to show that $\Ext^{i}_{\mathcal{X}}(L_{\mathcal{X}/A},  \tau_{\geq 1} \R \pi_{A,*} \mathcal{O}_Y )=0$ for $i \leq	 1$.  Since $L_{\mathcal{X}/A}$ is supported in non-positive cohomology degrees, while $\tau_{\geq 1} \R \pi_{A,*} \mathcal{O}_Y$ only in positive degrees, this vanishing is immediate for $i \leq 0$. For $i=1$, again by cohomology degree argument, it is the same as showing that the following Ext group is zero.
\begin{equation*}
\Ext^0_{\mathcal{X}}(\mathcal{H}^0(L_{\mathcal{X}/A}), \mathcal{H}^1( \tau_{\geq 1} \R \pi_{A,*} \mathcal{O}_Y)) \cong \Ext^0_{\mathcal{X}}(\Omega_{\mathcal{X/A}}, \R^1 \pi_{A,*} \mathcal{O}_Y)
\cong  \Ext^0_X(\Omega_X, \R^1 \pi_{A,*} \mathcal{O}_Y) ,
\end{equation*}
which is exactly one of the assumptions of the proposition. This finishes the proof of Claim.
\end{proof}


To show that $\Fib(a,p_{Y'})$ is non-empty and discrete, we will first construct a deformation of $X$ to $A'$ lifting $\mathcal{X}$, and then show that this deformation admits a morphism from the chosen deformation of $Y$ to $A'$ lifting $\mathcal{Y}$.

We now show the existence of a flat deformation $\mathcal{X'} \to \Spec(A')$ of $\mathcal{X} \to \Spec(A)$ across $\Spec(A) \subset \Spec(A')$. The obstruction of the existence of such a deformation the homomorphism $\ob(f,f^* D_A): L_{\mathcal{X}/A}[-1] \to \mathcal{O}_X[1]$. Since $\mathcal{Y}$ already has such a square-zero extension, the corresponding obstruction $\ob(g,g^* D_A) : L_{\mathcal{Y}/A}[-1] \to \mathcal{O}_X[1]$ is homotopic to zero. Furthermore, by Theorem \ref{thm:dag-sqzero-obs-compat}, these two obstructions are related via the following diagram (which is commutative in a specified manner):
\begin{equation*}
\xymatrix{
\pi_A^* L_{\mathcal{X}/A}[-1] \ar[r] \ar[dr]_{\pi_A^* \ob(f, f^* D_A)} & L_{\mathcal{Y}/A}[-1] \ar[d]^{ \ob(g, g^* D_A)} \\
& \mathcal{O}_Y[1]
}
\end{equation*}
In particular, $\pi_A^* \ob(f, f^* D_A)$ is nullhomotopic.  By Claim \ref{claim:extpullback}, $\ob(f, f^* D_A)$ is also nullhomotopic, so there exists a deformation $\mathcal{X}' \to \Spec(A')$  of $\mathcal{X} \to \Spec(A)$, as claimed above; we fix one such deformation.

Next, we show that the deformation $\mathcal{X}' \to \Spec(A')$ chosen above can be modified to allow for an $A'$-linear map $\pi_{A'}:\mathcal{Y}' \to \mathcal{X}'$ extending $\pi_A$. Let $D_X : L_{\mathcal{X}} \to \mathcal{O}_X[1]$ (resp. $D_Y:L_{\mathcal{Y}} \to \mathcal{O}_Y[1]$) be the derivation corresponding to the deformation $\mathcal{X}' \to \Spec(A')$ constructed above (resp. to the deformation $\mathcal{Y}' \to \Spec(A')$ that we started with). We obtain a diagram
\begin{equation}
\label{eq:cotdiag2}
\xymatrix{
g^* L_A \ar[d]^{D_A} \ar[r] & \pi^* L_{\mathcal{X}} \ar[r] \ar[d]^{D_X} & L_{\mathcal{Y}} \ar[d]^{D_Y} \\
g^* k[1] \ar@{=}[r] &	 \pi^* \mathcal{O}_X[1] \ar@{=}[r] & \mathcal{O}_Y[1]}
\end{equation}
where the square on the left commutes in a specified way by construction of $\mathcal{X}'$, and the outer square commutes in a specified way as $\mathcal{Y}' \to \Spec(A')$ lifts $g$. We must replace $D_X$ by a suitable map so that the square on the right also commutes in manner compatible with the other two squares. The failure of the commutativity of the square on the right is measured by the difference $\delta$ of the two paths $\pi^* L_{\mathcal{X}} \to \mathcal{O}_Y[1]$ in the square on the right. Since the outer square commutes, this obstruction $\delta$ factors as a map $\pi^* L_{\mathcal{X}/A} \to \mathcal{O}_Y[1]$. By Claim \ref{claim:extpullback}, this map is obtained as the pullback of a map $\delta':L_{\mathcal{X}/A} \to \mathcal{O}_X[1]$. Replacing $D_X$ with $D'' := D_X + \delta' \circ \can$ (where $\can:L_{\mathcal{X}} \to L_{\mathcal{X}/A}$ is the canonical map) as the middle vertical arrow in diagram \eqref{eq:cotdiag2} then makes all squares commute compatibly. This derivation $D''$ and the commutativity of the left hand square give rise to to a deformation $\mathcal{X}'' \to \Spec(A')$ of $\mathcal{X} \to \Spec(A)$ across $\Spec(A) \subset \Spec(A')$, while the commutativity of the right hand square give rise to the promised map $\pi_{A'}:\mathcal{Y}' \to \mathcal{X}''$. In particular, this proves that $\Fib(a,p_{Y'})$ is non-empty. 

Next, we check that $\Fib(a,p_{Y'})$ is discrete. For a point $(\mathcal{X'} \to \Spec(A'), \pi_{A'}:\mathcal{Y}' \to \mathcal{X}',\phi)$ of this groupoid, an automorphism $\sigma$ is given by an automorphism $\sigma$ of $\mathcal{X'}$ that commutes with $\pi_{A'}$ and $\phi$. Since topoi do not change under deformations, it suffices to prove that $\sigma$ acts as the identity on $\mathcal{X}'$. By definition, the induced action on $\pi^{-1} \mathcal{O}_{\mathcal{X}'}$ commutes with the map $\pi_{A'}^*:\pi^{-1} \mathcal{O}_{\mathcal{X}'} \to \mathcal{O}_{\mathcal{Y}'}$. Since the latter map is injective (which can be checked, for instance, by filtering both sides using powers of the maximal ideal of $A'$ to reduce to the known injectivity over $k$), it follows that $\sigma = \id$,  proving discreteness.

The conclusion of the preceding paragraphs is that the map $a$ from diagram \eqref{diag:defmor} is surjective with discrete fibers, and consequently that the map $q:\Def_\pi \to \Def_Y$ is formally smooth with discrete fibers.

Finally, we show that $\Fib(a,p_{Y'})$ is contractible when $X$ has no infinitesimal automorphisms. For $i = 1,2$, let $(\mathcal{X}_i' \to \Spec(A'),\pi_{i,A'}:\mathcal{Y}' \to \mathcal{X}_i',\phi_i)$ be two possibly distinct points of $\Fib(a,p_{Y'})$; we will show they are connected. First, we show that $\mathcal{X}_1$ and $\mathcal{X}_2$ are isomorphic as deformations of $\mathcal{X} \to \Spec(A)$ across $\Spec(A) \subset \Spec(A')$; this part will not use the assumption on $X$. Let $D_i:L_{\mathcal{X}} \to \mathcal{O}_X[1]$ be the derivation classifying the deformation $\mathcal{X}_i$. Then the information of $\pi_{i,A'}$ gives, for each $i$, a commutative diagram
\[ \xymatrix{
g^* L_A \ar[d]^{D_A} \ar[r] & \pi^* L_{\mathcal{X}} \ar[r] \ar[d]^{\pi^*D_i} & L_{\mathcal{Y}} \ar[d]^{D_Y} \\
g^* k[1] \ar@{=}[r] &	 \pi^* \mathcal{O}_X[1] \ar@{=}[r] & \mathcal{O}_Y[1].} \]
The commutativity shows that $\pi^* D_1$ and $\pi^* D_2$ are homotopic maps $\pi^* L_{\mathcal{X}} \to \mathcal{O}_Y[1]$: they are both homotopic to the composition $\pi^* L_{\mathcal{X}} \to L_{\mathcal{Y}} \stackrel{D_Y}{\to} \mathcal{O}_Y[1]$. By Claim \ref{claim:extpullback}, $D_1$ and $D_2$ are also homotopic, which proves that $\mathcal{X}_1$ and $\mathcal{X}_2$ are isomorphic as deformations of $\mathcal{X} \to \Spec(A)$ across $\Spec(A) \subset \Spec(A')$. Hence, to show that $\Fib(a,p_{Y'})$ is contractible, it suffices to check: given deformations $\mathcal{X}' \to \Spec(A')$ of $\mathcal{X} \to \Spec(A)$, and $\mathcal{Y}' \to \Spec(A')$ of $\mathcal{Y} \to \Spec(A)$, there exists at most one extension of $\pi_A:\mathcal{Y} \to \mathcal{X}$ to an $A'$-map $\pi_{A'}:\mathcal{Y}' \to \mathcal{X}'$. The $\infty$-groupoid of choices for such extensions is easily verified to be a torsor for 
\[ \Omega \Hom_{\mathcal{Y}}(\pi_A^* L_{\mathcal{X}/A},\mathcal{O}_Y[1]) \simeq \Hom_{\mathcal{Y}}(\pi_A^* L_{\mathcal{X}/A},\mathcal{O}_Y).\]
By Claim \ref{claim:extpullback}, the $\infty$-groupoid on the right is equivalent to $\Hom_{\mathcal{X}}(L_{\mathcal{X}/A},\mathcal{O}_X)$. By adjunction (see the proof of Theorem \ref{mainthm:defthy}), this $\infty$-groupoid is identified with $\Hom_X(L_X,\mathcal{O}_X)$ which, by assumption, is contractible.
\end{proof}

\begin{remark}
The methods used to show Proposition \ref{deformds} also show that (under the same hypotheses) one has a natural equivalence $e:\Def_{\id_X} \times_{\Def_X} \Def_\pi \simeq \Def_\pi \times_{\Def_Y} \Def_\pi$ where we view $\Def_{\id_X}$ as a space fibered over $\Def_X$ with fibers given by the automorphism groups of the corresponding deformation, and the map $e$ is given by $(a,b) \mapsto (a \circ b, b)$.
\end{remark}

\begin{remark}
The technique used in Proposition \ref{deformds} can be used to show the following refinement (under the same hypotheses): the map $q:\Def_\pi \to \Def_Y$ has a distinguished section $s$. Indeed, in the notation of the proof of Proposition \ref{deformds}, constructing $s$ amounts to constructing a canonical base point of $\Fib(a,p_{Y'})$; such a base point is provided by the deformation of $\mathcal{X}$ coming from the derivation $D_X:L_{\mathcal{X}} \to \mathcal{O}_X[1]$ whose pullback along $\pi^*$ is the derivation $\pi^* L_{\mathcal{X}} \to L_{\mathcal{Y}} \stackrel{D_Y}{\to} \mathcal{O}_Y[1]$. We leave the details to the reader.
\end{remark}

This next lemma relates infinitesimal automorphisms of the source and target of a given morphism under favorable conditions; this will be used in the sequel to move information about discreteness of the automorphism group of a stable variety to its covering stack (see Theorem \ref{mainthm:localagain}).

\begin{proposition}
\label{infautstabvarqcov}
Let $\pi:Y \to X$ be an essentially finitely presented morphism of noetherian Delinge-Mumford stacks (over some base ring $k$). Assume the following:
\begin{enumerate}
\item The map $\pi$ is \'etale on an open subset $U \subset Y$ that contains all the codimension $1$ points of $Y$.
\item The stack $Y$ satisfies Serre's $S_2$ condition at points of $Y - U$. 
\item The map $\pi$ satisfies $\pi_* \mathcal{O}_Y \simeq \mathcal{O}_X$.
\end{enumerate}
Then the infinitesimal automorphisms of $X$ and $Y$ coincide, i.e., there is a natural isomorphism $\Ext^0(L_X,\mathcal{O}_X) \simeq \Ext^0(L_Y,\mathcal{O}_Y)$ (where all cotangent complexes are computed relative to $k$).
\end{proposition}
\begin{proof}
The transitivity triangle for $\pi$ and the assumption that $\pi_* \mathcal{O}_Y \simeq \mathcal{O}_X$ give a long exact sequence
\[ 1 \to \Ext^0(L_\pi,\mathcal{O}_Y) \to \Ext^0(L_Y,\mathcal{O}_Y) \to \Ext^0(L_X,\mathcal{O}_X) \to \Ext^1(L_\pi,\mathcal{O}_Y) \to \dots \]
Thus, it suffices to show that $\Ext^i(L_\pi,\mathcal{O}_Y) = 0$ for $i \leq 1$. This follows by the exact same method used in the proof of Proposition \ref{defcms}; we omit the details.
\end{proof}

\subsection{The $\Q$-Gorenstein deformation theory}
\label{sec:qgordefthy}

We now return to the product map for moduli spaces of stable varieties.  Our goal is to show that the global product map $\Prod_{X,Y}$ is finite \'etale for two stable varieties $X$ and $Y$.  To understand the local behavior of this map, we cannot simply consider the local product map $\mathfrak{p}\mathrm{rod}_{X,Y}$ described in \S \ref{sec:defthyabstractprod} because Koll\'ar's condition restricts the allowable deformations on both sides. Instead, we introduce the {\em canonical covering stack} $Z^\can$ of a variety $Z$ for the reasons explained in \S \ref{sec:intro}. We simply remark here Proposition \ref{defthyindqcov} below, which equates $\Def_{X^\can \times Y^\can}$ with $\Def_{(X \times Y)^\can}$ under favorable assumptions, is proven using the results of \S \ref{sec:defmorphisms}.

\begin{definition} \label{def:canonicalcoveringstack}
Fix a field $k$, and let $X$ be an essentially finitely presented $\Q$-Gorenstein $k$-scheme satisfying Serre's condition $S_2$. Then we define its {\em canonical covering stack} $\pi:X^\can \to X$ by the formula
	\[ X^\can = [\uSpec(\oplus_{i \in \Z} \omega_X^{[i]})/\G_m] \]
where $\uSpec$ denotes the relative spectrum of a quasi-coherent $\mathcal{O}_X$-algebra, $\omega_X^{[i]}$ is the $i$-th reflexive power of the dualizing sheaf $\omega_X$, and the $\G_m$-action is given by the evident grading.
\end{definition}

We now describe some properties of canonical covering stacks.

\begin{lemma}
\label{cmsgorsch}
Fix a field $k$. Let $X$ be an essentially finitely presented $\Q$-Gorenstein $k$-scheme satisfying Serre's condition $S_2$, and let $\pi:X^\can \to X$ denote the structure morphism of the canonical cover. Then the following are true:
\begin{enumerate}
\item The stack $X^\can$ is an essentially finitely presented Artin $k$-stack satisfying Serre's condition $S_2$. If $k$ has characteristic $0$, then $X^\can$ is Deligne-Mumford.
\item The formation of $\pi$ commutes with relatively Gorenstein essentially finitely presented flat base changes $f:U \to X$.
\item The map $\pi$ is a coarse moduli space that is an isomorphism on the Gorenstein locus of $X$. 
\item The natural map $\mathcal{O}_X \to \R \pi_* \mathcal{O}_{X^\can}$ is an isomorphism.
\end{enumerate}
\end{lemma}

\begin{proof}
We first observe that the formation of $X^\can \to X$ commutes with localization on $X$ as the same is true for the sheaves $\omega_X$ and their reflexive powers. By the $\Q$-Gorenstein assumption, we may pick an integer $n > 0$ such that $\omega_X^{[n]}$ is actually a line bundle. After localizing on $X$ if necessary, we can pick an isomorphism $\mathcal{O}_X \simeq \omega_X^{[n]}$ defined by a section $s \in \omega_X^{[n]}$. Such a choice allows us to define the structure of a $\mathcal{O}_X$-algebra with a $\mu_n$-action on the coherent $\mathcal{O}_X$-module
\[ \mathcal{A} = \oplus_{i \in \Z/n} \omega_X^{[i]}  \]
in the obvious way: we view $\mathcal{A}$ as the quotient algebra of the algebra $\oplus_{i \in \Z} \omega_X^{[i]}$ by the equation $s = 1$, and the $\mu_n$-action corresponds to the induced $\Z/n$-grading. We set $Y = \uSpec(\mathcal{A})$ and observe that the natural map $Y \to X^\can$ is $\mu_n$-equivariant and therefore descends to a map
\[ g:[Y/\mu_n] \to X^\can.\]
We leave it to the reader to check that $g$ is an isomorphism; the key point is that the defining map $X^\can \to B(\G_m)$ factors through $B(\mu_n) \to B(\G_m)$ via the choice of $s$, and the scheme $Y$ is simply the fiber of the resulting map $X^\can \to B(\mu_n)$.  This presentation shows that $X^\can$ is an essentially finitely presented Artin $k$-stack if $X$ is so; if $k$ has characteristic $0$, then the presentation gives rise to a Deligne-Mumford stack since $\mu_n$ is discrete. To finish checking property (1), we observe that, by construction, the sheaves $\omega_X^{[i]}$ are $S_2$. Hence, the same is true for the scheme $Y$ and the stack $X^\can$. 

Property (2) follows from the next Lemma \ref{gorbc} and Lemma \ref{lem:reflex-pullback}. Indeed, if $\mathcal{F}$ is any coherent $\mathcal{O}_U$-module and $\mathcal{L}$ is a line bundle on $U$, then there is a natural isomorphism of $U$-stacks
\[  [\uSpec\big(\oplus_{i \in \Z} (\mathcal{F} \otimes \mathcal{L})^{[i]} \big) / \G_m]  \simeq [\uSpec\big(\oplus_{i \in \Z} \mathcal{F}^{[i]}\big) / \G_m]. \]
This observation applies here with $\mathcal{F} = f^* \omega_X$ and $\mathcal{L} = \omega_f$.

For property (3), we note that $\mathcal{O}_X$ is the sheaf of $\mu_n$-invariants of $\mathcal{A}$, which shows that $\pi$ is a coarse moduli space. The claim concerning the behavior over the Gorenstein locus follows from property (2).

For property (4), observe that the formula $X^\can = [Y/\mu_n]$ identifies the $\QCoh(X^\can)$ with the category $\QCoh(Y)^{\mu_n}$ of $\mu_n$-equivariant quasi-coherent sheaves on $Y$. The functor $\pi_*:\QCoh(X^\can) \to \QCoh(X)$ is then identified with the functor $\mu_n$-invariants which is exact because $\mu_n$ is linearly reductive, showing that $\R^i \pi_* \mathcal{O}_{X^\can} = 0$ for $i > 0$. Since the claim for $i = 0$ was already shown,  the result follows.
\end{proof}

The following lemma is used in the proof of property (2) above:

\begin{lemma}
\label{gorbc}
Let $f:U \to X$ be flat relatively Gorenstein morphism between essentially finitely presented schemes over some field $k$, and assume that $X$ admits a dualizing complex $\omega_X^\bullet$. Then there is a natural isomorphism of sheaves
\[ f^* \omega_X \otimes \omega_f \simeq \omega_U. \]
\end{lemma}

\begin{proof}
We normalize dualizing complexes so that the dualizing sheaf of a scheme sits inside the dualizing complex in homological degree equal to dimension of the scheme. After spreading out $U$ and $X$, we may assume that $f$ is a map between finite type separated $k$-schemes.  Choose compatible compactifications $U \subset \overline{U}$ and $X \subset \overline{X}$ together with a map $\overline{f}: \overline{U} \to \overline{X}$ extending $f$. By \cite[Theorem 5.4]{NeemanGD} (which applies because $\overline{U}$ and $\overline{X}$ are noetherian, and because $\R \overline{f}_*$ preserves coproducts by \cite[Lemma 1.4]{NeemanGD}), we have a canonical isomorphism
$$ \overline{f}^* \omega_{\overline{X}}^\bullet \otimes \omega_{\overline{f}} \simeq \omega_{\overline{U}}^\bullet. $$
Note that the dualizing complexes furnished by \cite{NeemanGD} agree with the usual ones for proper $k$-schemes.  Restricting to $U$, using the relatively Gorenstein assumption on $f$, and applying $\mathcal{H}^{-\dim(U)}$ now gives the desired claim.
\end{proof}

\begin{remark}
The only place where the characteristic $0$ assumption was used in Lemma \ref{cmsgorsch} was to conclude that $X^\can$ was a Deligne-Mumford stack rather than an Artin stack. This distinction is crucial to our proofs as Deligne-Mumford stacks have connective cotangent complexes, and the connectivity makes the proofs of Proposition \ref{defcms} and Proposition \ref{deformds} work.
\end{remark}

We record the following lemma here for use in Proposition \ref{defthyindqcov}.

\begin{lemma}
\label{depthprod}
Let $(R,\fm)$ and $(S,\fn)$ be two essentially finitely presented $k$-algebras over some algebraically closed field $k$, and let $(T,\fp)$ be the Zariski localization of $R \otimes_k S$ at the maximal ideal generated by $\fm$ and $\fn$. Then we have
\[ \depth_\fp(T) = \depth_\fm(R) + \depth_\fn(S) \]
\end{lemma}
\begin{proof}
The map $(R,\fm) \to (T,\fp)$ is an essentially finitely presented flat local homomorphism of noetherian local rings with fiber $T/\fm T \simeq S$. The addition formula for depth (see \cite[\S 21.C, Corollary 1]{MatCA}) now implies the claim.
\end{proof}

Finally, we show that the deformations of products of canonical covering stacks are the same as those for the canonical covering stack of a product provided there are no infinitesimal automorphisms in sight.

\begin{proposition}
\label{defthyindqcov}
Fix a field $k$ of characteristic $0$. Let $X$ and $Y$ be two essentially finitely presented $\Q$-Gorenstein $k$-schemes that are both Gorenstein in codimension $\leq 1$ and satisfy Serre's condition $S_2$. Assume that $X \times Y$ has no infinitesimal automorphisms. Then one has a natural equivalence of deformation functors $\Def_{X^\can \times Y^\can} \to \Def_{(X \times Y)^\can}$ as functors on $\Art_k$.
\end{proposition}
\begin{proof}
Let $\pi:X^\can \times Y^\can \to (X \times Y)^\can$ denote the canonical map, and let $\Def_\pi$ denote the deformation functor associated to $\pi$. Forgetting information defines morphisms $a:\Def_\pi \to \Def_{X^\can \times Y^\can}$ and $b:\Def_\pi \to \Def_{(X \times Y)^\can}$. We will show that each of these maps is an equivalence.

To show that the map $b$ is an equivalence, we apply Proposition \ref{defcms}. Let $U \subset X^\can \times Y^\can$ denote the locus where $\pi$ is \'etale. We will check that $U$ contains all the codimension $2$ points, and that $X^\can \times Y^\can$ satisfies Serre's condition $S_3$ on the complement of $U$. Since both conditions are local on $X \times Y$, we localize on the latter whenever necessary. Moreover, we freely identify points on a Delinge-Mumford stack and those on the coarse space.

Both $X^\can \times Y^\can$ and $(X \times Y)^\can$ are \'etale over $X \times Y$ at the Gorenstein points of the latter which includes all the codimension $1$ points. Hence, it suffices to check that $\pi$ is \'etale at the codimension $2$ points of $X^\can \times Y^\can$. We first observe that this last claim is clear if one of $X$ or $Y$ is Gorenstein itself: the formation of $X^\can \to X$ commutes with flat relatively Gorenstein base changes on $X$ by property (2) in Lemma \ref{cmsgorsch}.  Now a point of $X^\can \times Y^\can$ is given by a product $(x,y)$. Such a product has codimension $2$ if either both $x$ and $y$ have codimension $1$, or one has codimension $2$ and the other has codimension $0$. In either case, one of the factors appearing in the product is Gorenstein, and hence the map is \'etale by the preceding observation; this verifies that $U$ contains all the codimensions $\leq 2$ points. 

Next, we check Serre's condition. The same reasoning used above also shows that a point $(x,y)$ in the complement of $U$ defines points $x \in X$ and $y \in Y$ each with codimension $\geq 2$. Property (1) from Lemma \ref{cmsgorsch} implies that each of $X^\can$ and $Y^\can$ satisfy Serre's condition $S_2$. Hence, any point $(x,y) \in X^\can \times Y^\can - U$ automatically satisfies Serre's condition $S_3$ by Lemma \ref{depthprod}. By applying Proposition \ref{defcms}, we may now conclude that $b$ is an equivalence.

To show that the map $a$ is an equivalence, we apply Proposition \ref{deformds}. In order to apply this proposition, we first need to check that $(X \times Y)^\can$ has no infinitesimal automorphisms. This follows from Proposition \ref{infautstabvarqcov} applied to the map $(X \times Y)^\can \to X \times Y$ and the assumption that $X \times Y$ has no infinitesimal automorphisms. Next, we need to verify that $\mathcal{O}_{ (X \times Y)^\can } \stackrel{\simeq}{\to} \pi_* \mathcal{O}_{X^\can \times Y^\can}$ and that $\R^1 \pi_* \mathcal{O}_{X^\can \times Y^\can} = 0$. We may localize to assume that both $X$ and $Y$ are affine. Note that we have a commutative diagram
\[ \xymatrix{X^\can \times Y^\can \ar[r]^\pi \ar[d]^f & (X \times Y)^\can \ar[d]^g \\
				B(\G_m \times \G_m) \ar[r]^{p} & B(\G_m) } \]
where the vertical maps classify the defining quotient stack structure, and the map $p$ is induced by the multiplication map $\G_m \times \G_m \to \G_m$. Since we are working with affines, the vertical maps are affine faithfully flat and finitely presented maps and, thus, the corresponding pushforward functors are faithful. Now observe that the category $\QCoh(X^\can \times Y^\can)$ can be identified as the category of $(\G_m \times \G_m)$-equivariant objects on the fiber of $f$, and similarly for $(X \times Y)^\can$. It is then easy to see that the functor $\pi_*$ is identified with the functor of taking invariants under the antidiagonal $\G_m \subset \G_m \times \G_m$: it suffices to check the analogous claim for the map $p$ since pushing forward along the vertical maps is faithful, and then the claim follows from the basic formalism of classifying stacks. In particular, since $\G_m$ is linearly reductive, the higher direct images $\R^i \pi_* \mathcal{O}_{X^\can \times Y^\can}$ vanish for $i > 0$. The claim for $i = 0$ is an easy exercise in local coordinates, and we leave this to the reader.
\end{proof}

We now may put together all of our results on deformation theory to obtain our main theorem.

\begin{theorem} \label{mainthm:localagain}
Let $X$ and $Y$ be two stable varieties over a field $k$ of characteristic $0$. Then the natural map
\[ \Prod_{X,Y}:\mathcal{M}(X) \times \mathcal{M}(Y) \to \mathcal{M}(X \times Y) \]
is finite \'etale.
\end{theorem}
\begin{proof}
We first note that the morphism is well-defined by Proposition \ref{prop:product}. By Thoerem \ref{thm:proper_DM}, the stacks $\mathcal{M}(X)$ are proper Delign-Mumford stacks. Hence, by Zariski's main theorem for Deligne-Mumford stacks, it suffices to check that the map $\Prod_{X,Y}$ is \'etale at each point of $\mathcal{M}(X) \times \mathcal{M}(Y)$. Moreover, since each point of $\mathcal{M}(X) \times \mathcal{M}(Y)$ is given as a pair of stable varieties, we may without loss of generality restrict our attention to the canonical point of $\mathcal{M}(X) \times \mathcal{M}(Y)$ defined by $X$ and $Y$. In this case, by the main result of \cite{AbramovichHassett}, the deformation theory of $\mathcal{M}(X)$ at the point defined by $X$ is controlled by the functor $\Def_{X^\can}$ on $\Art_k$, and similarly for $Y$ and $X \times Y$. Hence, it suffices to check that the natural transformation
\[ \Def_{X^\can} \times \Def_{Y^\can} \to \Def_{ (X \times Y)^\can } \]
is an equivalence of functors on $\Art_k$. The result now follows from the factorization
\[ \Def_{X^\can} \times \Def_{Y^\can} \stackrel{a}{\to} \Def_{X^\can \times Y^\can} \stackrel{b}{\to} \Def_{ (X \times Y)^\can } \]
coupled with the fact that $a$ is an equivalence by Theorem \ref{mainthm:defthy}, while $b$ is an equivalence by Proposition \ref{defthyindqcov} (which applies by Lemma \ref{stablevarinfaut} and Lemma \ref{cmsgorsch}).
\end{proof}


\section{The global theory} \label{sec:globaltheory}

In this section, we prove Theorem \ref{mainthm:productirreducible} and Theorem
\ref{mainthm:global}. Note that, by the Lefschetz principle, we may assume that the
base field is the field of complex numbers, which we shall do from now on.

\subsection{Canonically polarized manifolds}

Recall that a \emph{canonically polarized manifold} is a compact complex manifold
whose canonical line bundle is ample; such a manifold is automatically a smooth
complex projective variety. We shall use two important theorems from
differential geometry -- Yau's theorem about the existence of K\"ahler-Einstein
metrics, and the Uhlenbeck-Yau theorem -- to show that any canonically polarized
manifold can be uniquely decomposed into a product of ``irreducible'' factors.

\begin{definition}
A canonically polarized manifold $X$ is called \emph{irreducible} if it
does not admit a nontrivial product decomposition $X \cong X_1 \times X_2$.
\end{definition}

It is easy to see that, in any product decomposition of a canonically polarized
manifold, every factor is again canonically polarized. By Chow's theorem, such a
decomposition is then automatically also a decomposition in the category of smooth
complex projective varieties.

\begin{theorem}
\label{theorem:main}
Let $X$ be a canonically polarized manifold. Then there is a product decomposition
\begin{equation*}
 X \cong X_1 \times \dotsb \times X_r
\end{equation*}
into irreducible canonically polarized manifolds, and this decomposition is unique
up to the order of the factors.
\end{theorem}

\subsection{Products of stable varieties}

The following statements are immediate consequences of Theorem \ref{theorem:main}.

\begin{corollary}
\label{cor:isomorphism}
If $Z$ is a canonically polarized manifold with irreducible decomposition $Z = Z_1 \times Z_2$, such that $\mathcal{M}(Z_1) \neq \mathcal{M}(Z_2)$ as components of the moduli stack of all stable varieties, then the product map
\begin{equation*}
 \Prod_{Z_1,Z_2} : \mathcal{M}(Z_1) \times \mathcal{M}(Z_2) \to \mathcal{M}(Z)
\end{equation*}
  is an isomorphism. Furthermore, the image of $\Prod_{X,Y}$ intersects $\mathcal{M}(Z)$ if and only if $X \in \mathcal{M}(Z_1)$ and $Y \in \mathcal{M}(Z_2)$ or $Y \in \mathcal{M}(Z_1)$ and $X \in \mathcal{M}(Z_2)$.
\end{corollary}

\begin{corollary}
\label{cor:exact_description}
If $X$ is a canonically polarized manifold with irreducible decomposition 
\begin{equation*}
X \cong \prod_{i=1}^r  \left( \prod_{j=1}^{n_i} X_{ij} \right)
\end{equation*}
such that $\mathcal{M}(X_{ij}) = \mathcal{M}(X_{i'j'})$ if and only if $i=i'$, then 
\begin{equation*}
 \mathcal{M}(X) \cong \prod_{i=1}^r \left[ \factor{\mathcal{M}(X_{i1})^{\times n_i}}{S_{n_i}} \right] ,
\end{equation*}
where the symmetric group $S_{n_i}$ acts on $\mathcal{M}(X_{i1})^{\times s_i}$ by
permuting the factors, and the quotient is taken in the stack sense.
\end{corollary}

We also have the following general formula for the degree of the fibers.

\begin{proposition}
\label{prop:product_map_fiber}
If $X$ and $Y$ are stable schemes, then the fiber of the map $\Prod_{X,Y} : \mathcal{M}(X) \times \mathcal{M}(Y) \to \mathcal{M}(X \times Y)$ over $X \times Y$ contains as many points as 
\begin{equation*}
\sum_{V \in \mathcal{M}(X), W \in \mathcal{M}(Y), V \times W \cong X \times Y } 
\left| \factor{\Aut(X \times Y)}{\Aut(V) \times \Aut(W)} \right|  
\end{equation*}
\end{proposition}

\subsection{Polystability of the tangent bundle}

The goal of this and the next section is to show that the the tangent bundle of a
canonically polarized manifold is polystable (with respect to the ample line bundle
$\omega_X$). 

\begin{theorem} 
\label{thm:canonically_polarized_tangent}
If $X$ is a canonically polarized manifold, then $\shT_X$ is polystable with
respect to $\omega_X$. More precisely, $\shT_X$ uniquely decomposes into a
direct sum of stable, pairwise non-isomorphic subbundles of slope $\mu(\shT_X)$.
\end{theorem}

We shall briefly recall the definition of stability and polystability. For a
torsion-free coherent sheaf $\mathcal{F}$ on $X$, we define the \emph{slope}
$\mu(\mathcal{F})$ with respect to the ample line bundle $\omega_X$ by the formula
\begin{equation} \label{eq:slope}
	\mu(\mathcal{F}) = 
		\frac{c_1(\mathcal{F}) \cdot c_1(\omega_X)^{\dim X-1}}{\rk \mathcal{F}},
\end{equation}
see for example \cite[Definition 1.2.11]{HD_LM_TG}. If $i \colon U \to X$ is the
inclusion of the open subset where $\mathcal{F}$ is locally free, then $\rk
\mathcal{F}$ means the rank of $i^* \mathcal{F}$, and $c_1(\mathcal{F})$ is defined
as the first Chern class of the line bundle $\det \mathcal{F} = i_* \det(i^*
\mathcal{F})$.

\begin{definition}
\label{defn:semi_poly_stable}
Let $\mathcal{F}$ be a torsion-free sheaf on a canonically polarized complex manifold
$X$.
\begin{enumerate}
\item \label{itm:semi_poly_stable:stable} $\mathcal{F}$ is \emph{stable} if for every
subsheaf $\mathcal{G} \subseteq \mathcal{F}$ with $0<\rk \mathcal{G}< \rk
\mathcal{F}$, one has $\mu(\mathcal{G}) < \mu(\mathcal{F})$,
\item \label{itm:semi_poly_stable:poly_stable} $\mathcal{F}$ is \emph{polystable} if
it is the direct sum of stable sheaves of the same slope.
\end{enumerate}
\end{definition}

The following simple lemma will be used in two places below.

\begin{lemma}  \label{lem:simple}
Let $\shE= \shE_1 \oplus \dotsb \oplus \shE_n$ be a polystable vector bundle, with
$\shE_i$ stable and pairwise non-isomorphic. If $\shE = \shF \oplus \shG$ for two
subsheaves $\shF, \shG \subseteq \shE$, then there is a subset $I \subseteq \{1,
\dotsc, n\}$ with the property that
\[
	\shF = \bigoplus_{i \in I} \shE_i \quad \text{and} \quad
		\shG = \bigoplus_{i \not\in I} \shE_i.
\]
\end{lemma}

\begin{proof}
Since $\shE_i$ are stable and pairwise non-isomorphic, \cite[Proposition
1.2.7]{HD_LM_TG} shows that we have
\begin{equation} \label{eq:EiEj}
	\Hom(\shE_i, \shE_j) = \begin{cases}
		\mathbb{C} &\text{for $i = j$,} \\
		0   &\text{for $i \neq j$.}
	\end{cases}
\end{equation}
Now consider the composition $i_{\shF} p_{\shF} \colon \shE \to \shE$ of the
projection $p_{\shF} \colon \shE \to \shF$ and the inclusion $i_{\shF} \colon \shF
\to \shE$. It is naturally represented by an $n \times n$-matrix; by \eqref{eq:EiEj} this
matrix is diagonal with entries in $\mathbb{C}$. Moreover, all diagonal entries are
either $0$ or $1$, on account of the identity $(i_{\shF} p_{\shF})(i_{\shF} p_{\shF})
= i_{\shF} p_{\shF}$. The same is true for the matrix representing $i_{\shG}
p_{\shG}$; since we have $i_{\shF} p_{\shF} + i_{\shG} p_{\shG} = \id_{\shE}$, the
assertion follows.
\end{proof}

\subsection{Differential geometry}

We shall now use two results from differential geometry to prove 
Theorem \ref{thm:canonically_polarized_tangent}. 

Let $(X, \omega)$ be a compact K\"ahler manifold; by a slight abuse of notation, we
shall use the symbol $\omega$ both for the K\"ahler metric and for its associated
real closed $(1,1)$-form. We can use the same formula as in \eqref{eq:slope} to
define the slope of torsion-free coherent sheaves on $X$, replacing $c_1(\omega_X)$
by the cohomology class $\lbrack \omega \rbrack \in H^2(X, \mathbb{R})$ of the K\"ahler
form. We therefore have the notion of stability and polystability with respect to
$\omega$. 

Now recall that the K\"ahler metric $\omega$ is called \emph{K\"ahler-Einstein} if
$\Ric \omega= \lambda \omega$ for some real number $\lambda$. Here $\Ric \omega$ is
the Ricci curvature form of $\omega$, or equivalently the Chern curvature $\sqrt{-1}
\Theta(\det \shT_X, \det \omega)$ of the naturally induced metric on the
holomorphic line bundle $\det \shT_X$; the constant $\lambda$ is called the
\emph{scalar Ricci curvature} of $\omega$. 

\begin{theorem}[\cite{AT_EDT, YST_OTR}]
If $X$ is a canonically polarized complex manifold, and $\lambda<0$ a real number,
then $X$ admits a unique K\"ahler-Einstein metric with scalar Ricci curvature
$\lambda$.  
\end{theorem}

In the following, we shall normalize the K\"ahler-Einstein metric $\omega$ on a
canonically polarized manifold $X$ by taking its scalar Ricci curvature equal to $- 2
\pi$; in other words, we shall assume that $\Ric \omega = - 2 \pi  \omega$. With this
convention, 
\begin{equation} \label{eq:Kahler_Einstein_constant:chern_equals_metric}
c_1(\omega_X, (\det \omega)^{-1}) = \omega , 
\end{equation} where $c_1(\omega_X,
(\det \omega)^{-1})$ is the Chern form of the induced metric on the canonical line
bundle. Indeed, 
\begin{equation*} 
c_1(\omega_X, (\det \omega)^{-1}) = - c_1(\det \shT_X, \det \omega) 
= - \frac{\sqrt{-1}}{2 \pi} \Theta(\det \shT_X, \det \omega)
=  - \frac{1}{2 \pi} \Ric \omega = \omega,  \end{equation*}
where the second equality is by \cite[Example 4.4.8.i]{HD_CG}. This ensures that the
slope with respect to the ample line bundle $\omega_X$ is the same as the slope with
respect to the K\"ahler form $\omega$; in particular, the two notions of stability
(and polystability) coincide.

\begin{proposition}[{e.g.,\cite[Definitions 4.B.1 and 4.B.11]{HD_CG}}]
Let $(X, \omega)$ be a compact K\"ahler-Einstein manifold. Then the
induced metric on the holomorphic tangent bundle $\shT_X$ is
Hermite-Einstein.
\end{proposition}

Recall that a Hermitian metric $h$ on a holomorphic vector bundle $\mathcal{E}$ is
is called \emph{Hermite-Einstein} if 
\begin{equation}
\sqrt{-1} \Lambda_{\omega} \Theta(\mathcal{E},h) =  \lambda \id_{\mathcal{E}};
\end{equation}
here $\Lambda_{\omega}$ is the metric contraction on the space of complex-valued
two-forms, induced by the K\"ahler metric $\omega$. The Uhlenbeck-Yau theorem relates
this differential-geometric condition back to algebraic geometry.

\begin{theorem}[\cite{UK_YST_OTE}]
\label{theorem:metric_polystable}
On a compact K\"ahler manifold $(X,\omega)$, a holomorphic vector bundle
admits a Hermite-Einstein metric  if and only if it is polystable with respect to $\omega$.
\end{theorem}

Here is the proof that the tangent bundle of a canonically polarized manifold is
polystable. 

\begin{proof}[Proof of Theorem \ref{thm:canonically_polarized_tangent}]
Since $X$ is canonically polarized, it admits a unique K\"ahler-Einstein metric
$\omega$ with $\Ric \omega = - 2 \pi \omega$. The induced metric on the tangent
bundle is Hermite-Einstein, and by the Uhlenbeck-Yau theorem, $\shT_X$ is
polystable with respect to $\omega$, hence also polystable with respect to
$\omega_X$. This means that we have a decomposition
\[
	\shT_X = \shE_1 \oplus \dotsb \oplus \shE_n
\]
into stable subbundles $\shE_i$ of slope
\[
	\mu(\shE_i) = \mu(\shT_X) = -\frac{c_1(\omega_X)^{\dim X}}{\dim X} < 0.
\]
The argument in \cite[Lemma~1.3]{BA_CM} now shows that the $\shE_i$ must be pairwise
non-isomorphic: indeed if $\shE_i \simeq \shE_j$ for $i \neq j$, then $\shE_i$ would
carry a flat connection, which is not possible because $\mu(\shE_i) < 0$. Finally,
the uniqueness of the decomposition follows from Lemma~\ref{lem:simple}.
\end{proof}

\subsection{Proof of the theorem}
\label{sec:proof}

We now come to the proof of Theorem \ref{theorem:main}. It is easy to see (by
induction on the dimension) that every canonically polarized
manifold has at least one decomposition 
\[	
	X \cong X_1 \times \dotsm \times X_r
\]
into irreducible canonically polarized manifolds $X_i$. It remains to show
that this decomposition is unique, up to the order of the factors. For this, it is
clearly enough to prove that any two product decompositions of a canonically
polarized manifold admit a common refinement. This, in turn, is implied by the
following special case.

\begin{lemma}
Let $X \cong Y \times Z \cong Y' \times Z'$ be two product decompositions of a canonically
polarized manifold. Then there is a common refinement $X \cong W_1 \times W_2 \times
W_3 \times W_4$, with the property that
\begin{equation} \label{eq:refinement}
	Y \cong W_1 \times W_2, \quad
	Z \cong W_3 \times W_4, \quad
	Y' \cong W_1 \times W_3, \quad
	Z' \cong W_2 \times W_4.
\end{equation}
\end{lemma}

\begin{proof}
By Theorem \ref{thm:canonically_polarized_tangent}, the tangent bundle of $X$ is
polystable, and in fact, decomposes uniquely as 
\begin{equation} \label{eq:splitting_TX}
	\shT_X = \shE_1 \oplus \dotsb \oplus \shE_n
\end{equation}
with $\shE_i$ stable and pairwise non-isomorphic. To simplify the notation, we put
\[
	\shE(I) = \bigoplus_{i \in I} \shE_i
\]
for any subset $I \subseteq \{1, \dotsc, n\}$. The decompositions $X \cong Y \times Z
\cong Y' \times Z'$ of the manifold $X$ induce decompositions $\shT_X = p_Y^* \shT_Y
\oplus p_Z^* \shT_Z = p_{Y'}^* \shT_{Y'} \oplus p_{Z'}^* \shT_{Z'}$ of its tangent
bundle. It then follows from Lemma~\ref{lem:simple} that the set $\{1, \dotsc, n\}$
can be partitioned into four disjoint subsets $I_1$, $I_2$, $I_3$, and $I_4$, in such
a way that
\[
	p_Y^* \shT_Y = \shE(I_1 \cup I_2), \quad 
	p_Z^* \shT_Z = \shE(I_3 \cup I_4), \quad 
	p_{Y'}^* \shT_{Y'} = \shE(I_1 \cup I_3), \quad 
	p_{Z'}^* \shT_{Z'} = \shE(I_2 \cup I_4).
\]

Let $\pi \colon \Xt \to X$ be the universal covering space of $X$; note that
$\Xt$ will usually be non-compact. The splitting $\shT_X = \shE(I_1) \oplus \shE(I_2)
\oplus \shE(I_3) \oplus \shE(I_4)$ lifts to a splitting of $\mathcal{T}_{\Xt}$, and
therefore induces a decomposition
\[
	\Xt \cong M_1 \times M_2 \times M_3 \times M_4
\]
into integral submanifolds of the foliations $\pi^* \shE(I_k)$, according to
\cite[Theorem~A]{BA_CM}. By the same result, the fundamental group $G = \pi_1(X)$ acts
compatibly on each factor $M_k$, in such a way that the natural action on $\Xt$ is
diagonal. In particular, this means that we have an embedding of groups
\[
	G \to \Aut(M_1) \times \Aut(M_2) \times \Aut(M_3) \times \Aut(M_4),
\]
where $\Aut(M_k)$ denotes the group of biholomorphic automorphisms of the complex
manifold $M_k$. Let us denote the preimage of $\Aut(M_k)$ under this embedding by the
letter $G_k$, the preimage of $\Aut(M_k) \times \Aut(M_{\ell})$ by the letter
$G_{k\ell}$, and so on. We claim that $G \cong G_1 \times G_2 \times G_3 \times G_4$. 

To prove this claim, we observe that $M_1 \times M_2$ is a simply connected
integral submanifold of the foliation $\pi^* p_Y^* \shT_Y$, and must therefore be
the universal covering space of $Y$; consequently, $\pi_1(Y)$ embeds into $\Aut(M_1)
\times \Aut(M_2)$. The same is of course true in the other three cases. Since we
have $\pi_1(X) \cong \pi_1(Y) \times \pi_1(Z) \cong \pi_1(Y') \times \pi_1(Z')$
compatibly with the above decompositions, it follows that
\[
	G \cong G_{12} \times G_{34} \cong G_{13} \times G_{24}.
\]
From this, it is easy to deduce that $G \cong G_1 \times G_2 \times G_3 \times G_4$.

To conclude the proof, we define $W_k = M_k / G_k$. We then have $X \cong W_1 \times
W_2 \times W_3 \times W_4$, and so each $W_k$ must be a compact complex manifold;
because $X$ is canonically polarized, each $W_k$ is also canonically polarized, and
therefore a smooth complex projective variety by Chow's theorem. It is clear from the
construction that \eqref{eq:refinement} is satisfied, and so the lemma is proved.
\end{proof}

\newpage
\appendix
\section{Review of some derived algebraic geometry}
\label{sec:dag}

In this appendix, we summarize the deformation theory relevant to us, using the language of derived algebraic geometry. Our primary goal is to explain certain functorialities in the usual deformation-obstruction theory of varieties by interpreting everything in terms of {\em derivations} in the derived category\footnote{We hasten to remark that all statements written here are well-known to the experts, and have been written down simply to provide a convenient reference.}. We do so by first discussing the deformation-obstruction theory for derivations (see \S \ref{derobs} and \S \ref{compat}), then explaining how to realize the theory of square-zero extensions as a special case of the theory of derivations (see \S \ref{sqzeroext}), and then finally recording the corresponding statements for the deformation-obstruction theory for square-zero extensions (see \S \ref{obssqzero}). The format adopted is that of short numbered paragraphs, each one discussing an algebraic problem and its solution first, and then stating the corresponding scheme-theoretic result (with references). To avoid mentioning derived Deligne-Mumford stacks in the {\em statements} of various theorems below, we impose flatness hypotheses in the statements. We hope that this sacrifice of generality will make the statements more readily accessible. Our primary references will be \cite{LurieDAG} and \cite{IllusieCC1}, though occasionally we refer to \cite[Chapter 8]{LurieHA} as well; we freely use the language of \cite{LurieHT} and \cite[Chapter 1]{LurieHA}.

\subsection{Conventions} We use the term $\infty$-groupoid when referring to a mapping space in an $\infty$-category. Given an $\infty$-category $\mathcal{C}$ and objects $X,Y \in \mathcal{C}$, we let $\Hom_{\mathcal{C}}(X,Y)$ denote the  $\infty$-groupoid of maps in $\mathcal{C}$ between $X$ and $Y$; we drop the subscript $\mathcal{C}$ from the notation when the category is clear from context. Fix a Grothendieck abelian category $\mathcal{A}$, and consider the stable $\infty$-category $\mathcal{D}$ of (unbounded) chain complexes over $\mathcal{A}$ with its usual t-structure; see \cite[Section 1.3.5]{LurieHA} for more. Given an object $K \in \mathcal{D}$ and an integer $j$, the complex $K[j]$ denotes the complex $K$ with homological degree increased by $j$. We freely identify $\mathcal{D}_{\leq 0}$ with the $\infty$-category of simplicial objects in $\mathcal{A}$ via the Dold-Kan correspondence, and we use the term {\em connective} to refer to such chain complexes.  We sometimes denote the shift functor $K \mapsto K[-1]$ by $\Omega$. Since $\mathcal{A}$ has enough injectives, for any two (bounded above) complexes $C$ and $D$ in $\mathcal{D}$, the $\infty$-groupoid $\Hom(C,D)$ of maps $C \to D$ in $\mathcal{D}$ can also be realized as
\[ \Hom(C,D) = \tau_{\leq 0} \Hom^\bullet(C,\tilde{D}). \]
where $D \to \tilde{D}$ is a quasi-isomorphism between $D$ and a complex $\tilde{D}$ of $K$-injectives, and $\Hom^\bullet(C,\tilde{D})$ is the mapping chain complex in the usual sense of homological algebra; note that $\Hom(C,D)$ is a simplicial abelian group. If $f:M \to N$ is a morphism in a stable $\infty$-category $\mathcal{C}$, then $N/M$ denotes the pushout of $f$ along $M \to 0$ and is called the {\em homotopy cokernel} of $f$. Dually, the pullback of $f$ along $0 \to N$ is called the {\em homotopy kernel} of $f$. Note that $\Omega N$ is simply the homotopy kernel of $0 \to N$. We denote by 
\[ M \to N \to N/M \]
the exact triangle defined by $f$ in the homotopy category of $\mathcal{C}$; we exclude the boundary map $N/M \to M[1]$ in our depiction of the exact triangle simply for notational convenience. 

\subsection{The basic setup} We will work in the setting of derived algebraic geometry provided by the $\infty$-category $\SAlg_k$ of simplicial commutative $k$-algebras over some fixed (ordinary) base ring $k$ rather than any more sophisticated variants; the fullsubcategory of $\SAlg_k$ spanned by discrete $k$-algebras is ordinary and will be denoted $\Alg_k$. All tensor products are assumed to be derived and relative to $k$ unless otherwise specified; in particular, the subcategory $\Alg_k \subset \SAlg_k$ is not closed under the $\otimes$-products unless $k$ is a field. The reference \cite{LurieHA} works in the setting of $E_\infty$-rings rather than $\SAlg_k$ (the two coincide with $\Q \subset k$, up to connectivity constraints), but can also be adapted to work for $\SAlg_k$; we choose to ignore this issue when referring to \cite{LurieHA} below.

\subsection{Stable $\infty$-categories of modules and their functorialities} Given an $A \in \SAlg_k$, one can define a stable $\infty$-category $\Mod(A)$ of $A$-modules which comes equipped with a natural $\otimes$-product structure. When $A$ is discrete, this $\infty$-category realizes the derived category of the abelian category of $A$-modules as its homotopy category; we stress that $\Mod(A)$ is {\em not} the ordinary category of $A$-modules when $A$ is discrete. The association $A \mapsto \Mod(A)$ obeys the expected functorialities. For example, if $f:A \to B$ is a map in $\SAlg_k$, there is an extension of scalars functor $\Mod(A) \to \Mod(B)$ induced by tensoring with $B$, and there is a restriction functor $\Mod(B) \to \Mod(A)$ obtained by remembering only the $A$-structure. These functors are adjoint: given $M \in \Mod(A)$ and $N \in \Mod(B)$, there exist functorial equivalences
\[ \Hom_{\Mod(A)}(M,N) \simeq \Hom_{\Mod(B)}(M \otimes_A B,N). \] 

\subsection{The cotangent complex} \label{cotcomplex} Let $A$ be an ordinary $k$-algebra.  The cotangent complex $L_A$ of $A$ relative to $k$, sometimes denoted by $L_{A/k}$ when the ring $k$ is not clear from context, is an object in $\Mod(A)$ constructed as follows:  pick an $A_\bullet \in \SAlg_k$ with an equivalence $f:A_\bullet \to A$ such that each $A_n$ is a free $k$-algebra. Then the $A$-modules $\Omega^1_{A_n/k} \otimes_{A_n} A$ assemble naturally to form a simplical $A$-module. The corresponding object in $\Mod(A)$ is called the cotangent complex $L_A$. This construction can be generalized to an arbitrary $A \in \SAlg_k$, and also generalizes to simplicial rings in an arbitrary topos. Note that in each case the cotangent complex is actually {\em connective}; non-connective cotangent complexes arise if one works with Artin stacks, but these do not concern us.  A non-abelian derived functor approach to the cotangent complex can be found in \cite{QuillenCRC}; the book \cite{IllusieCC1} is the original source, and contains the details of the above construction.

\subsection{Derivations}\label{defder} Let $A \in \SAlg_k$, and let $M$ be an $A$-module. A $k$-linear derivation $A \to M$ is, by definition, a $k$-algebra section of the projection map $A \oplus M \to A$, where $A \oplus M$ is given an $A$-algebra structure via the usual $A$-action on $M$, and with $M \subset A \oplus M$ being a square-zero ideal; one can easily check that this recovers the usual notion when $A$ and $M$ are discrete. Let $\Der_k(A,M)$ denote the $\infty$-groupoid of all $k$-linear derivations $A \to M$ (we drop the subscript $k$ from the notation when the base ring $k$ is fixed). By construction of the cotangent complex, one has a derivation $d:A \to L_A$. It is a theorem that this derivation is the universal one:
\begin{theorem}
\label{thm:dag-der}
With notation as above, composition with $d$ induces a functorial equivalence
\begin{equation} \label{defdereq} \Hom(L_A,M) \simeq \Der(A,M). \end{equation}
\end{theorem}
The case when $k$, $A$, and $M$ are discrete can be found in \cite[Corollary II.1.2.4.3]{IllusieCC1}, while the case where $M$ is allowed to be a complex can be found in \cite[Proposition II.1.2.6.7]{IllusieCC1}. In \cite{LurieDAG}, the cotangent complex is {\em defined} using the preceding property (see the discussion preceding \cite[Remark 3.2.8]{LurieDAG}). To see that this construction agrees with the Illusie's construction as explained in \S \ref{cotcomplex}, one observes that there is a map from Lurie's cotangent complex to Illusie's. Moreover, this map is an isomorphism when the algebra is free (use \cite[Lemma 3.2.13]{LurieDAG}), and therefore always an isomorphism by passage to free resolutions. A model-categorical approach to the universal properties of the cotangent complex can be found in \cite{QuillenCRC}.

\subsection{Functoriality of derivations} \label{derfunc} Let $f:A \to B$ be a map in $\SAlg_k$, and let $N$ be a $B$-module. Using formula \eqref{defdereq}, we see
\[ \Der(A,N) \simeq \Hom_A(L_A,N) \simeq \Hom_B(L_A \otimes_A B,N). \]
In other words, the natural derivation $A \to L_A \to L_A \otimes_A B$ is the universal derivation from $A$ into a $B$-module.

\subsection{The transitivity triangle} \label{dertrans} Let $f:A \to B$ be a map in $\SAlg_k$. Composing the natural derivation $B \to L_B$ with $f$ defines a derivation $A \to L_B$ and, by \S \ref{derfunc},  a map $L_A \otimes_A B \to L_B$. One can show that this map induces an identification
\[ L_B / (L_A \otimes_A B) \simeq L_{B/A} \]
in the stable $\infty$-category of $B$-modules. At the level of triangulated categories, this gives rise to an exact triangle (see \cite[Proposition 3.2.12]{LurieDAG}) and \cite[Proposition II.2.1.2]{IllusieCC1}), called the transitivity triangle, of the form
\[ L_A \otimes_A B \to L_B \to L_{B/A}. \]
The associated boundary map $L_{B/A}[-1] \to L_A \otimes_A B$ is called the {\em Kodaira-Spencer class} of $f$, and is denoted by $\kappa(f)$.

\subsection{Base change} \label{derbasechange} Let $A,B \in \SAlg_k$. Then the composite map
\[ L_A \otimes B \to  L_{A \otimes B} \to L_{A \otimes B/B}\]
is an isomorphism. One way to see this is by passage to free resolutions; see \cite[Proposition 3.2.9]{LurieDAG} for the corresponding scheme-theoretic statement. Alternately, one can also prove this drectly using \S \ref{dertrans}. This fact (for $k$, $A$, and $B$ discrete) can be found in \cite[Proposition II.2.2.1]{IllusieCC1}.

\subsection{Extending derivations across morphisms } \label{derobs} Let $f:A \to B$ be a map in $\SAlg_k$, and let $M$ be an $A$-module. Given a derivation $D:A \to M$, a natural question to ask is if the following diagram can be filled
\begin{equation} \label{derfuncdiag} \xymatrix{ A \ar[r]^D \ar[d] & M \ar[d] \\
			  B & M \otimes_A B } \end{equation}
using a $k$-linear derivation $B \to M \otimes_A B$. Formally speaking, we are asking the following: given a $k$-algebra section $s_D:A \to A \oplus M$ of the projection map $A \oplus M \to A$, when can the diagram
\[ \xymatrix{ A \ar[r]^{s_D} \ar[d] & A \oplus M \ar[d] \\
			  B & B \oplus M \otimes_A B }\]
be filled with a $k$-algebra homomorphism $B \to B \oplus M \otimes_A B$ splitting the projection $B \oplus M \otimes_A B \to B$? By Theorem \ref{thm:dag-der} and its functoriality in $A$ and $M$, the preceding question is equivalent to asking if
\[ \xymatrix{ L_A \ar[r] \ar[d] & M \ar[d] \\
			  L_B & M \otimes_A B }\]
can be filled using a $B$-module map $L_B \to M \otimes_A B$; here the horizontal map $L_A \to M$ is the map defined by $D$. We may refine this diagram to obtain
\[ \xymatrix{ L_A \ar[r] \ar[d] & M \ar[d] \\
			  L_A \otimes_A B \ar[r] \ar[d] & M \otimes_A B \ar@{=}[d] \\
			  L_B & M \otimes_A B. }\]
Thus, requiring the existence of a $k$-linear derivation $B \to M \otimes_A B$ filling diagram \eqref{derfuncdiag} is equivalent to requiring that the map $L_A \otimes_A B \to M \otimes_A B$ induced by the original derivation $A \to M$ factors through $L_A \otimes_A B \to L_B$. Moreover, the space of all possible ways of filling the diagram above is tautologically the homotopy-fiber of 
\[ \Hom_B(L_B,M \otimes_A B) \simeq \Der_k(B,M \otimes_A B) \to \Der_k(A,M) \simeq \Hom_B(L_A \otimes_A B, M \otimes_A B) \]
over the point corresponding to $D:A \to M$. Using the rotated transitivity triangle
\[ L_{B/A}[-1] \to L_A \otimes_A B \to L_B \]
we see that such a factorization exists if and only if the induced map $L_{B/A}[-1] \to M \otimes_A B$ is trivial. We denote this last map by $\ob(f,D)$ and refer to it as the obstruction to extending $D$ across $f$. When $\ob(f,D)$ vanishes, the description as a homotopy-fiber above shows that $\infty$-groupoid of all possible ways of filling in diagram \eqref{derfuncdiag} by a $k$-linear derivation $B \to M \otimes_A B$ (together with the relevant homotopy) is naturally a torsor under $\Hom(L_{B/A},M \otimes_A B)$; in particular, the set of all possible extensions (up to homotopy) of $D$ across $f$ is a torsor under $\pi_0(\Hom(L_{B/A},M \otimes_A B))$.  Generalizing this discussion to simplicial rings in a topos, and then specializing to the case of Deligne-Mumford stacks, we obtain:

\begin{theorem}
\label{thm:dag-obs-der}
Let $f:X \to Y$ be a flat morphism of Deligne-Mumford stacks, and let $D_Y:L_Y  \to \mathcal{M}$ be a derivation on $Y$ into a connective quasi-coherent complex $\mathcal{M}$ of $\mathcal{O}_Y$-modules. Then the obstruction to the existence of a derivation $D_X:L_X \to f^* \mathcal{M}$ commuting with $f^* D_Y$ is the map
\[ \ob(f,f^* D_Y): L_{X/Y}[-1] \stackrel{\kappa(f)}{\to} f^* L_Y \stackrel{f^* D_Y}{\to} f^*\mathcal{M} \]
where the map $\kappa(f)$ is the Kodaira-Spencer class for $f$. When $\ob(f,f^* D_S)$ vanishes, the set of all pairs $(D_X:L_X \to f^*\mathcal{M},H:D_X \to f^* D_Y)$ (where $D_X$ is a derivation, and $H$ is a homotopy expressing the commutativity of $D_X$ with $f^*D_Y$) up to homotopy is a torsor for $\Ext^0_X(L_{X/Y},f^*\mathcal{M})$. Moreover, the $\infty$-groupoid of automorphisms of any such pair is equivalent to $\Hom_X(L_{X/Y},f^*\mathcal{M}[-1])$ 
\end{theorem}

Theorem \ref{thm:dag-obs-der} can be easily checked in the case where $\mathcal{M}$ is discrete (in which case the automorphisms mentioned at the end of Theorem \ref{thm:dag-obs-der} are all trivial). The general case comes at no extra cost, and the additional flexibility of allowing $\mathcal{M}$ to be a genuine complex instead of a sheaf will allow us later to treat the obstruction theory of square-zero extensions as a special case of the obstruction theory of derivations as presented in Theorem \ref{thm:dag-obs-der}; see \S \ref{obssqzero}, especially Theorem \ref{thm:dag-sqzero-obs}, which is essentially equivalent to the case of Theorem \ref{thm:dag-obs-der} when $\mathcal{M}$ is taken to be a sheaf placed in homological degree $1$. Theorem \ref{thm:dag-obs-der} is not stated explicitly in \cite{IllusieCC1} for the simple reason that Illusie chooses not to develop an obstruction theory for derivations.

\subsection{Compatibilities for obstructions with respect to a morphism} \label{compat} Let $A \stackrel{f}{\to} B \stackrel{g}{\to} C$ be composable maps in $\SAlg_k$. Given an $A$-module $M$ and a derivation $D:A \to M$, the discussion in \S \ref{derobs} produces maps 
\[ \ob(f,D):L_{B/A}[-1] \to M \otimes_A B \quad \textrm{and} \quad \ob(g \circ f,D):L_{C/A}[-1] \to M \otimes_A C\] 
which are obstructions to extending the derivation $D$ across $f$ and $g\circ f$ respectively. These obstructions are compatible in the sense that the following diagram commutes:
\[ \xymatrix{ L_{B/A} \otimes_B C \ar[rr]^-{\ob(f,D) \otimes_B C} \ar[d] & & M \otimes_A B \otimes_B C[1] \simeq M \otimes_A C[1] \ar[d] \\
			  L_{C/A} \ar[rr]^-{\ob(g\circ f,D)} & & M \otimes_A C [1]. } \]
This compatibility follows formally from the commutativity of the following diagram (which we leave to the reader to verify)
\[ \xymatrix{ L_{B/A}[-1] \otimes_B C \ar[r] \ar[d] & L_A \otimes_A B \otimes_B C \ar[r] \ar[d]^{\simeq} & L_B \otimes_B C \ar[d] \\
			  L_{C/A}[-1] \ar[r] \ar[rd] & L_A \otimes_A C \ar[r] \ar[d] & L_C \\
				& M \otimes_A C. & & } \]
Here the first row is the exact triangle induced by tensoring the (rotated) transitivity triangle for $A \to B$ with $C$, while the second row is the transitivity triangle for $A \to C$. Generalizing this discussion to simplicial rings in a topos and then specializing to the case of Deligne-Mumford stacks, we obtain:

\begin{theorem}
\label{thm:dag-obs-der-compat}
Let $g:Y \to S$ and $f:X \to S$ be flat morphisms of Deligne-Mumford stacks, and let $\pi:Y \to X$ be an $S$-morphism. Let $D_S:L_S \to \mathcal{M}$ be a derivation on $S$ into a connective quasi-coherent complex $\mathcal{M}$ of $\mathcal{O}_S$-modules. Then the obstructions $\ob(f,f^* D_S)$ and $\ob(g,g^* D_S)$ (as defined in Theorem \ref{thm:dag-obs-der}) are compatible in the sense that
\[ \xymatrix{  \pi^* L_{X/S}[-1] \ar[rr]^{\pi^*} \ar[rrd]_{\pi^* \ob(f,f^* D_S)} & &  L_{Y/S}[-1] \ar[d]^{\ob(g,g^* D_S)} \\
				& & \pi^* f^* \mathcal{M} \simeq g^* \mathcal{M} } \]
is commutative, i.e., a canonical homotopy expressing the commutativity exists. In particular, the cohomology classes
\[ \ob(f,f^* D_S) \in \Ext^1_X(L_{X/S},f^*\mathcal{M}) \quad \textrm{and} \quad \ob(g,g^* D_S) \in \Ext^1_Y(L_{Y/S},g^* \mathcal{M}) \] 
map to the same class in 
\[ \Ext^1_Y(\pi^* L_{X/S}, \pi^* f^* \mathcal{M}) \simeq \Ext^1_Y(\pi^* L_{X/S}, g^* \mathcal{M}).\]
under the natural maps.

\end{theorem}

This theorem is not stated explicitly in \cite{LurieDAG} or \cite{IllusieCC1}, but follows from the formula given in Theorem \ref{thm:dag-obs-der} and the functoriality in $A$ of Theorem \ref{thm:dag-der}.

\subsection{Square-zero extensions} \label{sqzeroext} A square-zero extension of an $A \in \SAlg_k$ by an $A$-module $M$ is, by definition, a derivation $A \to M[1]$. In order to see the connection with the usual definition, we use the following construction: a derivation $D:A \to M[1]$ gives rise to, by definition, a $k$-algebra section $s_D:A \to A \oplus M[1]$ of the projection map $A \oplus M[1] \to A$. Hence, we can form the following pullback square
\[ \xymatrix{ A^D \ar[r] \ar[d] & A \ar[d]^{s_D} \\
              A \ar[r]^-{s_0} & A \oplus M[1] } \]
where $s_0$ is the map associated to the $0$ derivation $A \to M[1]$, i.e., the standard section. When $A$ and $M$ are discrete, one calculates that $A^D$ is also discrete, and that the algebra map $A^D \to A$ is surjective with kernel a square-zero ideal isomorphic to $M$, justifying the choice of terminology. There also exists an intrinsic definition of square-zero extensions, and it is a theorem of Lurie (see the $k = \infty$ and $n = 0$ case of \cite[Theorem 8.4.1.26]{LurieHA}) that the preceding construction produces all such square-zero extensions when certain connectivity assumptions on $M$ (harmless for applications we have in mind) are satisfied. Hence, we will often abuse notation and denote a square-zero extension of $A$ by $M$ via an algebra map $\tilde{A} \to A$ with kernel $M$. Generalizing this discussion to rings in a topos and then specializing to the case of Deligne-Mumford stacks, we obtain:

\begin{theorem}
\label{thm:dag-sq-zero}
Let $X$ be a Deligne-Mumford stack over some base scheme $S$, and let $\mathcal{I} \in \QCoh(X)$. Then the construction above defines an equivalence between the groupoid $\Hom(L_{X/S},\mathcal{I}[1])$ and the groupoid of all square-zero extensions of $X$ by $\mathcal{I}$ over $S$.
\end{theorem}

The notion of ``square-zero extensions'' used in Theorem \ref{thm:dag-sq-zero} coincides with that of \cite[\S III.1]{IllusieCC1}. Theorem \ref{thm:dag-sq-zero} can be deduced from $k = 0$ case of \cite[Proposition 3.3.5]{LurieDAG}, and can also be found in \cite[Theorem III.1.2.3]{IllusieCC1}.

\subsection{Extending square-zero extensions across morphisms and compatibilities} \label{obssqzero} Let $f:A \to B$ be a map in $\SAlg_k$, and let $M$ be an $A$-module. Given a square-zero extension $\tilde{A} \to A$ of $A$ by $M$, a natural question to ask is if there exists a square-zero extension $\tilde{B} \to B$ of $B$ by $M \otimes_A B$ and a map $\tilde{A} \to \tilde{B}$ such that the following diagram commutes and is a pushout\footnote{Note that when $A, M$ and $B$ are discrete and $f$ is flat, the rings $\tilde{A}$ and $\tilde{B}$ are necessarily discrete with the map $\tilde{A} \to \tilde{B}$ being flat by the local flatness criterion. Hence, the preceding question generalizes the ordinary deformation-theoretic question of extending square-zero deformations of the target of a flat morphism of Deligne-Mumford stacks to that of the source.}:
\[ \xymatrix{  \tilde{A}  \ar[r] \ar[d] & A \ar[d] \\
				\tilde{B} \ar[r] & B } \]
Using our definition of square-zero extensions from \S \ref{sqzeroext}, this question is equivalent to the following: given a derivation $D:A \to M[1]$, when does there exist a derivation $D':B \to M \otimes_A B[1]$ such that the following diagram commutes?
\[ \xymatrix{ A \ar[r]^D \ar[d] & M[1] \ar[d] \\
			 B \ar[r]^-{D'} & M \otimes_A B[1]. } \]
Using the obstruction theory explained in \S \ref{derobs}, we find that such an extension exists if and only if $\ob(f,D)$ vanishes. When $\ob(f,D)$ does vanish, the $\infty$-groupoid of all possible extensions is naturally a torsor under $\Hom(L_{B/A},M \otimes_A B[1])$; in particular, the set of all possible extensions (up to homotopy) of $\tilde{A} \to A$ across $f$ is a torsor under $\pi_0(\Hom(L_{B/A},M \otimes_A B[1]))$.  Generalizing this discussion to rings in a topos and then specializing to the case of Deligne-Mumford stacks, we obtain:

\begin{theorem}
\label{thm:dag-sqzero-obs}
Let $f:X \to S$ be a flat morphism of Deligne-Mumford stacks. Fix a quasi-coherent $\mathcal{O}_S$-module $\mathcal{I}$, and a square-zero thickening $S \hookrightarrow S'$ of $S$ by $\mathcal{I}$ classified by a derivation $D_S:L_S \to \mathcal{I}[1]$. The obstruction to finding a square-zero thickening $X \hookrightarrow X'$ of $X$ by $f^* \mathcal{I}$ lying above $S \hookrightarrow S'$ (via a flat map $X' \to S'$) is the map
\[ \ob(f,f^* D_S):L_{X/S}[-1] \stackrel{\kappa(f)}{\to} f^* L_S \stackrel{f^* D_S}{\to} f^*\mathcal{I}[1] \]
where the map $\kappa(f)$ is the Kodaira-Spencer class of $f$. When $\ob(f,f^* D_S)$ vanishes, the set of all pairs $(X \hookrightarrow X', f':X' \to S')$ (where $X'$ is a thickening of $X$ by $f^* \mathcal{I}$, and $f'$ is a flat map deforming $f$) up to isomorphism is a torsor for $\Ext^1_X(L_{X/S},f^*\mathcal{I})$. Moreover, the group of automorphisms of any such pair is canonically identified with $\Ext^0_X(L_{X/S},f^* \mathcal{I})$.
\end{theorem}

Theorem \ref{thm:dag-sqzero-obs} follows from \cite[Proposition 8.4.2.5]{LurieHA}. Everything except the formula for $\ob(f,f^* D_S)$ can also be found in \cite[Proposition III.2.3.2]{IllusieCC1},  and the formula can be found in \cite[\S III.2.3.4]{IllusieCC1}. Finally, combining Theorem \ref{thm:dag-sqzero-obs} with Theorem \ref{thm:dag-obs-der-compat}, we obtain:

\begin{theorem}
\label{thm:dag-sqzero-obs-compat}
Let $g:Y \to S$ and $f:X \to S$ be flat morphisms of Deligne-Mumford stacks, and let $\pi:Y \to X$ be an $S$-morphism. Fix a quasi-coherent $\mathcal{O}_S$-module $\mathcal{I}$, and square-zero thickening $S \hookrightarrow S'$ of $S$ by $\mathcal{I}$ classified by a derivation $D_S:L_S \to \mathcal{I}[1]$. Then the obstructions $\ob(f,f^* D_S)$ and $\ob(g,g^* D_S)$ (as defined in Theorem \ref{thm:dag-sqzero-obs}) are compatible in the sense that 
\[ \xymatrix{  \pi^* L_{X/S}[-1] \ar[rr]^{\pi^*} \ar[rrd]_{\pi^* \ob(f,f^* D_S)} & &  L_{Y/S}[-1] \ar[d]^{\ob(g,g^* D_S)} \\
				& & \pi^* f^* \mathcal{I}[1] \simeq g^* \mathcal{I}[1] } \]
is commutative, i.e., a canonical homotopy expressing the commutativity exists. In particular, the cohomology classes 
\[ \ob(f,f^* D_S) \in \Ext^2_X(L_{X/S},f^*\mathcal{I}) \quad \textrm{and} \quad \ob(g,g^* D_S) \in \Ext^2_Y(L_{Y/S},g^* \mathcal{I}) \] 
map to the same class in 
\[ \Ext^2_Y(\pi^* L_{X/S}, \pi^* f^* \mathcal{I}) \simeq \Ext^2_Y(\pi^* L_{X/S}, g^* \mathcal{I}).\]
under the natural maps.
\end{theorem}

\newpage
\bibliography{modprodbib}

\begin{thebibliography}{BCHM10}

\bibitem[AH10]{AbramovichHassett}
Dan Abramovich and Brendan Hassett.
\newblock Stable varieties with a twist.
\newblock In {\em Classification of Algebraic Varieties}, EMS Series of
  Congress Reports, pages 1--38. European Mathematical Society, Zurich, 2010.

\bibitem[Ale02]{AV_CMI}
Valery Alexeev.
\newblock Complete moduli in the presence of semiabelian group action.
\newblock {\em Ann. of Math. (2)}, 155(3):611--708, 2002.

\bibitem[AP09]{AV_PR_ECO}
Valery Alexeev and Rita Pardini.
\newblock Explicit compactifications of moduli spaces of campedelli and burniat
  surfaces.
\newblock {\em arXiv:0901.4431}, 2009.

\bibitem[Aub76]{AT_EDT}
Thierry Aubin.
\newblock \'{E}quations du type {M}onge-{A}mp\`ere sur les vari\'et\'es
  k\"ahleriennes compactes.
\newblock {\em C. R. Acad. Sci. Paris S\'er. A-B}, 283(3):Aiii, A119--A121,
  1976.

\bibitem[BCHM10]{BC_CP_HCD_MJ_EOM}
Caucher Birkar, Paolo Cascini, Christopher~D. Hacon, and James McKernan.
\newblock Existence of minimal models for varieties of log general type.
\newblock {\em J. Amer. Math. Soc.}, 23(2):405--468, 2010.

\bibitem[Bea00]{BA_CM}
Arnaud Beauville.
\newblock Complex manifolds with split tangent bundle.
\newblock In {\em Complex analysis and algebraic geometry}, pages 61--70. de
  Gruyter, Berlin, 2000.

\bibitem[Cat86]{CF_CCO}
F.~Catanese.
\newblock Connected components of moduli spaces.
\newblock {\em J. Differential Geom.}, 24(3):395--399, 1986.

\bibitem[GKKP11]{GD_KS_KSJ_PT_DF}
Daniel Greb, Stefan Kebekus, S{\'a}ndor~J. Kov{\'a}cs, and Thomas Peternell.
\newblock Differential forms on log canonical spaces.
\newblock {\em Publ. Math. Inst. Hautes \'Etudes Sci.}, (114):87--169, 2011.

\bibitem[Gro65]{GA_EGA_IV_II}
A.~Grothendieck.
\newblock \'{E}l\'ements de g\'eom\'etrie alg\'ebrique. {IV}. \'{E}tude locale
  des sch\'emas et des morphismes de sch\'emas. {II}.
\newblock {\em Inst. Hautes \'Etudes Sci. Publ. Math.}, (24):231, 1965.

\bibitem[Hac04]{HP_CM}
Paul Hacking.
\newblock Compact moduli of plane curves.
\newblock {\em Duke Math. J.}, 124(2):213--257, 2004.

\bibitem[Har94]{HR_GD}
Robin Hartshorne.
\newblock Generalized divisors on {G}orenstein schemes.
\newblock In {\em Proceedings of {C}onference on {A}lgebraic {G}eometry and
  {R}ing {T}heory in honor of {M}ichael {A}rtin, {P}art {III} ({A}ntwerp,
  1992)}, volume~8, pages 287--339, 1994.

\bibitem[Has99]{HB_SLS}
Brendan Hassett.
\newblock Stable log surfaces and limits of quartic plane curves.
\newblock {\em Manuscripta Math.}, 100(4):469--487, 1999.

\bibitem[HKT06]{HP_KS_TJ_COT}
Paul Hacking, Sean Keel, and Jenia Tevelev.
\newblock Compactification of the moduli space of hyperplane arrangements.
\newblock {\em J. Algebraic Geom.}, 15(4):657--680, 2006.

\bibitem[HKT09]{HP_KS_TJ_SPT}
Paul Hacking, Sean Keel, and Jenia Tevelev.
\newblock Stable pair, tropical, and log canonical compactifications of moduli
  spaces of del {P}ezzo surfaces.
\newblock {\em Invent. Math.}, 178(1):173--227, 2009.

\bibitem[HL97]{HD_LM_TG}
Daniel Huybrechts and Manfred Lehn.
\newblock {\em The geometry of moduli spaces of sheaves}.
\newblock Aspects of Mathematics, E31. Friedr. Vieweg \& Sohn, Braunschweig,
  1997.

\bibitem[Huy05]{HD_CG}
Daniel Huybrechts.
\newblock {\em Complex geometry}.
\newblock Universitext. Springer-Verlag, Berlin, 2005.
\newblock An introduction.

\bibitem[Ill71]{IllusieCC1}
Luc Illusie.
\newblock {\em Complexe cotangent et d\'eformations. {I}}.
\newblock Lecture Notes in Mathematics, Vol. 239. Springer-Verlag, Berlin,
  1971.

\bibitem[Kaw07]{KM_IOA}
Masayuki Kawakita.
\newblock Inversion of adjunction on log canonicity.
\newblock {\em Invent. Math.}, 167(1):129--133, 2007.

\bibitem[Kol95]{KollarFlatness}
J{\'a}nos Koll{\'a}r.
\newblock Flatness criteria.
\newblock {\em J. Algebra}, 175(2):715--727, 1995.

\bibitem[Kol10]{KollarModuliSurvey}
J{\'a}nos Koll{\'a}r.
\newblock Moduli of varieties of general type.
\newblock Available at \url{http://arxiv.org/abs/1008.0621}, 2010.

\bibitem[Kol12]{KJ_SO}
J{\'a}nos Koll{\'a}r.
\newblock {\em Singularities of the Minimal Model Program}.
\newblock 2012.

\bibitem[Kov09]{KovacsYPG}
S{\'a}ndor~J. Kov{\'a}cs.
\newblock Young person's guide to moduli of higher dimensional varieties.
\newblock In {\em Algebraic geometry---{S}eattle 2005. {P}art 2}, volume~80 of
  {\em Proc. Sympos. Pure Math.}, pages 711--743. Amer. Math. Soc., Providence,
  RI, 2009.

\bibitem[Lee00]{LY_ACO}
Yongnam Lee.
\newblock A compactification of a family of determinantal {G}odeaux surfaces.
\newblock {\em Trans. Amer. Math. Soc.}, 352(11):5013--5023, 2000.

\bibitem[Liu12]{LW_SD}
Wenfei Liu.
\newblock Stable degenerations of surfaces isogenous to a product ii.
\newblock {\em Trans. Amer. Math. Soc.}, 364:2411--2427, 2012.

\bibitem[Lur04]{LurieDAG}
Jacob Lurie.
\newblock {\em Derived Algebraic Geometry}.
\newblock PhD thesis, M.I.T., 2004.
\newblock Available at \url{http://www.dspace.mit.edu/handle/1721.1/30144}.

\bibitem[Lur09]{LurieHT}
Jacob Lurie.
\newblock {\em Higher topos theory}, volume 170 of {\em Annals of Mathematics
  Studies}.
\newblock Princeton University Press, Princeton, NJ, 2009.

\bibitem[Lur10]{LurieICM}
Jacob Lurie.
\newblock Moduli {P}roblems for {R}ing {S}pectra.
\newblock Available at \url{http://math.harvard.edu/~lurie}, 2010.

\bibitem[Lur11]{LurieHA}
Jacob Lurie.
\newblock Higher {A}lgebra.
\newblock Available at \url{http://math.harvard.edu/~lurie}, 2011.

\bibitem[Mat80]{MatCA}
Hideyuki Matsumura.
\newblock {\em Commutative algebra}, volume~56 of {\em Mathematics Lecture Note
  Series}.
\newblock Benjamin/Cummings Publishing Co., Inc., Reading, Mass., second
  edition, 1980.

\bibitem[MM64]{MatsusakaMumford}
T.~Matsusaka and D.~Mumford.
\newblock Two fundamental theorems on deformations of polarized varieties.
\newblock {\em Amer. J. Math.}, 86:668--684, 1964.

\bibitem[Nee96]{NeemanGD}
Amnon Neeman.
\newblock The {G}rothendieck duality theorem via {B}ousfield's techniques and
  {B}rown representability.
\newblock {\em J. Amer. Math. Soc.}, 9(1):205--236, 1996.

\bibitem[Qui70]{QuillenCRC}
Daniel Quillen.
\newblock On the (co-) homology of commutative rings.
\newblock In {\em Applications of {C}ategorical {A}lgebra ({P}roc. {S}ympos.
  {P}ure {M}ath., {V}ol. {XVII}, {N}ew {Y}ork, 1968)}, pages 65 -- 87. Amer.
  Math. Soc., Providence, R.I., 1970.

\bibitem[Rol10]{RS_CM}
S{\"o}nke Rollenske.
\newblock Compact moduli for certain {K}odaira fibrations.
\newblock {\em Ann. Sc. Norm. Super. Pisa Cl. Sci. (5)}, 9(4):851--874, 2010.

\bibitem[UY86]{UK_YST_OTE}
K.~Uhlenbeck and S.-T. Yau.
\newblock On the existence of {H}ermitian-{Y}ang-{M}ills connections in stable
  vector bundles.
\newblock {\em Comm. Pure Appl. Math.}, 39(S, suppl.):S257--S293, 1986.
\newblock Frontiers of the mathematical sciences: 1985 (New York, 1985).

\bibitem[Vak06]{VR_MLI}
Ravi Vakil.
\newblock Murphy's law in algebraic geometry: badly-behaved deformation spaces.
\newblock {\em Invent. Math.}, 164(3):569--590, 2006.

\bibitem[Vie95]{VE_QPM}
Eckart Viehweg.
\newblock {\em Quasi-projective moduli for polarized manifolds}, volume~30 of
  {\em Ergebnisse der Mathematik und ihrer Grenzgebiete (3) [Results in
  Mathematics and Related Areas (3)]}.
\newblock Springer-Verlag, Berlin, 1995.

\bibitem[vO05]{VanOpstallModProd}
Michael~A. van Opstall.
\newblock Moduli of products of curves.
\newblock {\em Arch. Math. (Basel)}, 84(2):148--154, 2005.

\bibitem[vO06]{VOMA_SDIP}
Michael van Opstall.
\newblock Stable degenerations of surfaces isogenous to a product of curves.
\newblock {\em Proc. Amer. Math. Soc.}, 134(10):2801--2806 (electronic), 2006.

\bibitem[Yau78]{YST_OTR}
Shing~Tung Yau.
\newblock On the {R}icci curvature of a compact {K}\"ahler manifold and the
  complex {M}onge-{A}mp\`ere equation. {I}.
\newblock {\em Comm. Pure Appl. Math.}, 31(3):339--411, 1978.

\end{thebibliography}
\bibliographystyle{amsalpha}

\end{document}